\documentclass[11pt]{article}

\usepackage[T1]{fontenc}
\usepackage[utf8]{inputenc}
\usepackage[english]{babel}

\usepackage[dvipsnames,svgnames]{xcolor}

\usepackage{amsmath,amssymb,amsfonts,amsthm,enumerate,enumitem,upgreek,fancyhdr,verbatim}
\usepackage{bbm}
\usepackage{booktabs}
\usepackage{algorithm,algpseudocode}

\usepackage{seqsplit}
\usepackage{xstring}

\usepackage[a4paper]{geometry}
\setlist{topsep=1pt}

\usepackage{url}
\usepackage{hyperref}
\hypersetup{
    colorlinks = true,
    linkcolor = OliveDrab,
    citecolor = BrickRed,
    urlcolor = blue
}
\usepackage[capitalise,noabbrev,nameinlink]{cleveref}

\theoremstyle{definition}

\newtheorem{theorem}{Theorem}[section]
\newtheorem{definition}[theorem]{Definition}
\newtheorem{proposition}[theorem]{Proposition}
\newtheorem{corollary}[theorem]{Corollary}

\newtheorem{lemma}[theorem]{Lemma}
\newtheorem{fact[theorem]}{Fact}
\newtheorem{remark}[theorem]{Remark}
\newtheorem{example}[theorem]{Example}
\newtheorem{assumption}[theorem]{Assumption}
\newtheorem{fact}[theorem]{Fact}

\title{\Tema\thanks{FJAA, RC and CLP were partially supported by Grant PID2022-136399NB-C21 funded by ERDF/EU and by MICIU/AEI/10.13039/501100011033. The research of HHB was supported by the Natural Sciences and Engineering Research Council of Canada. CLP was supported by Grant PREP2022-000118 funded by MICIU/AEI/10.13039/501100011033 and by ``ESF Investing in your future''.}}
\author{Francisco J. Arag\'on-Artacho\thanks{Department of Mathematics, University of Alicante, \textsc{Spain}. E-mail:~\href{mailto:francisco.aragon@ua.es}{francisco.aragon@ua.es}.}
        \and Heinz H. Bauschke\thanks{Mathematics, University of British Columbia, Kelowna, \textsc{Canada}. E-mail:~\href{mailto:heinz.bauschke@ubc.ca}{heinz.bauschke@ubc.ca}.}
        \and Rub\'en Campoy\thanks{Department of Mathematics, University of Alicante, \textsc{Spain}. E-mail:~\href{mailto:ruben.campoy@ua.es}{ruben.campoy@ua.es}.}
        \and C\'esar L\'opez-Pastor\thanks{Department of Mathematics, University of Alicante, \textsc{Spain}. E-mail:~\href{mailto:cesar.lopez@ua.es}{cesar.lopez@ua.es}.}
}
\date{}

\usepackage{tikz}
\usepackage{tikz-cd}
\usetikzlibrary{decorations.markings,arrows}
\usetikzlibrary{shapes}
\usepackage{adjustbox}

\usepackage{exercise}

\usepackage{stmaryrd}

\usepackage{subcaption}
\usepackage{mathtools}
\usepackage{pgf}
\usepackage[font={small}]{caption}

\usepackage[overload]{empheq}

\usepackage{bm}

\usepackage{algorithm}

\newcommand{\Tema}{Graph splitting methods: \\ Fixed points and strong convergence for linear subspaces}

\DeclareMathAlphabet\mbc{OMS}{cmsy}{b}{n}

\newcommand{\zer}{\operatorname{zer}}
\newcommand{\Fix}{\operatorname{Fix}}
\newcommand{\ran}{\operatorname{ran}}
\newcommand{\spa}{\operatorname{span}}
\newcommand{\rk}{\operatorname{rank}}

\newcommand{\Id}{\operatorname{Id}}

\newcommand{\Diag}{\operatorname{Diag}}
\newcommand{\Lap}{\operatorname{Lap}}
\newcommand{\Inc}{\operatorname{Inc}}
\newcommand{\bs}{\boldsymbol}

\newcommand{\Hi}{\mathcal{H}}

\newcommand{\R}{\mathbb{R}}
\newcommand{\N}{\mathbb{N}}
\newcommand{\E}{\mathcal{E}}
\newcommand{\Z}{ Z}

\newcommand{\tto}{\rightrightarrows}
\usepackage{scalerel}
\newcommand{\ot}[1]{#1_{\scaleto{\otimes}{4pt}}}

\newcommand{\norm}[1]{\left\lVert #1\right\rVert}

\newcommand{\crefpart}[2]{%
  \hyperref[#2]{\namecref{#1}~\labelcref*{#1}~\ref*{#2}}%
}
\crefname{equation}{}{}
\crefname{enumi}{}{}

\usepackage{todonotes}
\reversemarginpar

\let\epsilon\varepsilon

\usepackage{etoolbox}
\apptocmd{\thebibliography}{\setlength{\itemsep}{1.9pt}}{}{}

\begin{document}

\maketitle

\begin{abstract}
    In this paper, we develop a general analysis for the fixed points of the operators defining the graph splitting methods from [\emph{SIAM J. Optim.}, 34 (2024), pp. 1569--1594] by Bredies, Chenchene and Naldi. We particularize it to the case where the maximally monotone operators are normal cones of closed linear subspaces and provide an explicit formula for the limit points of the graph splitting schemes. We exemplify these results on some particular algorithms, unifying in this way some results previously derived as well as obtaining new ones.
\end{abstract}

\section{Introduction}\label{sec: 1}

In this work, we focus on inclusion problems of the form
\begin{equation}\tag{P}\label{eq: P}
    \text{find }x\in\Hi\text{ such that }0\in\sum_{i=1}^n A_i(x),
\end{equation}
where $A_1,\ldots,A_n:\Hi\tto\Hi$ are set-valued maximally monotone operators and $\Hi$ is a Hilbert space.
This inclusion problem naturally arises in convex optimization when minimizing the sum of proper lower semicontinuous convex functions $f_i:\Hi\to(-\infty,+\infty]$, $i=1,\ldots,n$. Indeed, under standard constraint qualifications (see , e.g., \cite[Corollary 16.50]{bauschke}), the optimality condition for
\[
\min_{x\in\Hi}\; f_1(x)+\cdots+f_n(x)
\]
reduces to $0\in\sum_{i=1}^n \partial f_i(x)$, which is a particular instance of \cref{eq: P} with $A_i=\partial f_i$.

Provided that $\sum_{i=1}^n A_i$ is maximally monotone, the \emph{proximal point algorithm} can be implemented (see, e.g.,~\cite[Section 2.1]{RockaProx}), which transforms the inclusion problem into an equivalent fixed point problem $x=T(x)$, where $T$ is the resolvent of $\sum_{i=1}^n A_i$. However, this approach requires the calculation of its resolvent, which might be as difficult as solving the initial inclusion problem~\cref{eq: P}. For this reason, \emph{splitting methods} present a suitable alternative to this methodology: an algorithmic procedure where the resolvent of each operator $A_i$ is computed separately.

The first splitting method developed for finding a zero in the sum of $n=2$ maximally monotone operators was the celebrated \emph{Douglas--Rachford splitting}, developed by Lions and Mercier in~\cite{LM79}. Forty years later, Ryu proposed in~\cite{ryu20} the first splitting method able to handle $n=3$ operators without relying on a product space reformulation. The algorithm has the following structure:
\begin{align*}[left=\empheqlbrace]
x_1^{k+1}&=J_{A_1}(v_1^k),\\
x_2^{k+1}&=J_{A_2}(x_1^{k+1}+v_2^k),\\
x_3^{k+1}&=J_{A_3}(x_1^{k+1}+x_2^{k+1}-v_1^k-v_2^k),\\
v_1^{k+1}&=v_1^k+\theta(x_3-x_1),\\
v_2^{k+1}&=v_2^k+\theta(x_3-x_2),
\end{align*}
for some starting points $v_1^0,v_2^0\in\Hi$ and some relaxation parameter $\theta\in(0,1)$. This method, as well as Douglas--Rachford, is \emph{frugal}, meaning that it computes the resolvent $J_{A_i}$ of each operator $A_i$ only once per iteration.

Observe that five variables are used in Ryu's splitting method, of which three variables, $x_i$, store each computation of $J_{A_i}$, and two extra variables, $v_j$, are required for calculating the aforementioned resolvents. The former are called \emph{shadow variables} (or \emph{resolvent variables}) and the latter \emph{governing variables}. The number of governing variables in the splitting method is referred as \emph{lifting}. For $n=3$ operators, Ryu proved that the minimal number of governing variables for a frugal splitting method is $2$, giving rise to the notion of \emph{minimal lifting}. In subsequent years, several frugal splitting methods with minimal lifting arose to solve general inclusion problems for $n\geq 3$. Malitsky and Tam generalized Ryu's result in~\cite{malitsky2023resolvent}, showing that the minimal lifting is $n-1$. They also introduced a splitting algorithm with minimal lifting, described as follows:
\begin{align*}[left=\empheqlbrace]
x_1^{k+1} &=  J_{A_1}(v_1^k),\\
x_i^{k+1} &=   J_{A_i}( x_{i-1}^{k+1}+ v_i^k -v_{i-1}^k), \quad i= 2,\ldots, n-1,\\
x_n^{k+1} &=  J_{A_n}( x_{1}^{k+1}+x_{n-1}^{k+1}-v_{n-1}^k ),\\
v_i^{k+1} &=   v_i^k + \theta (x_{i+1}^{k+1}-x_{i}^{k+1}), \qquad i= 1,\ldots, n-1,
\end{align*}
where $\theta\in(0,1)$.

All these algorithms share a similar structure: inside each resolvent, a linear combination of shadow and governing variables
is computed. Not only that, but the reliance between their variables can be depicted by a directed graph and a subgraph. Indeed, let $\{1,\ldots,n\}$ be the nodes of the graph, representing the shadow
variables $x_i$, and let $(i,j)$ be an edge if $x_i$ is used as an input of the resolvent of~$A_j$. The involvement of the governing variables in the computation of the shadow variables can be modeled
using the subgraph. When the subgraph is a tree (which has $n-1$ edges), we can associate each edge to a governing variable. An incident edge to some node $i$ in the subgraph implies that the corresponding governing variable is used to update the shadow variable $x_i$.

Let us illustrate in \cref{fig: Ryu} the graph and subgraph defining Ryu's splitting method. The graph has edges $(1,3)$ and $(2,3)$ because the shadow variables $x_1$ and $x_2$ are used in the computation of $x_3$, while the edge $(1,2)$ shows that $x_1$ is used in the computation of $x_2$. On the other hand, since the governing variable $v_1$ is used in the computation of $x_1$ and $x_3$, the edge $(1,3)$ belongs to the subgraph, which we associate with $v_1$. Likewise, variable $v_2$ is used for updating $x_2$ and $x_3$, so we associate it with the edge $(2,3)$ in the subgraph.

\begin{figure}[h!]
    \centering
\begin{subfigure}{0.49\textwidth}
    \begin{center}
        \begin{tikzcd}[cells={nodes={draw=black, circle}}]
            1  \arrow[r] \arrow[rr, bend left] & 2 \arrow[r]  & 3
        \end{tikzcd}
    \end{center}
    \end{subfigure}
    \begin{subfigure}{0.49\textwidth}
    \begin{center}
        \begin{tikzcd}[cells={nodes={draw=black, circle}}]
            1  \arrow[rr, bend left, "v_1"] & 2 \arrow[r, "v_2"'] & 3
        \end{tikzcd}
    \end{center}
    \end{subfigure}

    \caption{Graph (left) and subgraph (right) defining Ryu's splitting}
    \label{fig: Ryu}
    \end{figure}

The key development of Bredies, Chenchene and Naldi in~\cite{graph-drs} was showing the reverse implication: each pair of algorithmic graph and connected spanning subgraph (see \cref{sec: 2.3}) yields a globally convergent frugal splitting algorithm with minimal lifting. This is valid even when the subgraph is not a tree, in which case the relationships between the governing and shadow variables are determined through an onto decomposition of the Laplacian matrix of the subgraph. In this way, their procedure unifies the convergence analysis in one general framework that encompasses all the presented schemes that were studied independently, and permits to derive new algorithms. For more details, in addition to~\cite{graph-drs}, we refer the reader to~\cite{graph-fb}, where the graph-based framework was extended to derive forward-backward algorithms for solving inclusion problems also involving cocoercive operators.

In the same spirit, in this work we use this graph-based algorithmic framework to unify the results provided in~\cite{bauschke-fixedpoints} and to obtain new ones, without the need of developing an ad hoc convergence analysis for each algorithm.  In~\cite{bauschke-fixedpoints}, the authors studied the fixed points of two operators that define the iterations of the (reduced and original) preconditioned proximal point algorithms of \cite{degenerate-ppp}. Different splitting algorithms, such as Douglas--Rachford, Ryu or Malitsky--Tam, can be derived as an instance of a preconditioned proximal point algorithm for an adequate construction of the operator and the preconditioner. Making use of this, exact expressions of the limit points of these algorithms were also obtained in \cite{bauschke-fixedpoints} for the case where the operators $A_i$ are normal cones of closed linear subspaces, proving also strong convergence. However, the analysis was carried out without the unified graph framework and, consequently, the results for each algorithm were obtained separately. The advantage of the techniques presented in this work lies in the achievement of general results based on the graph-based splitting framework, which encompass those obtained in~\cite{bauschke-fixedpoints}. In addition, this permits us to easily derive similar analyses for other algorithms, as the ones recently developed in \cite{graph-drs,graph-fb}.

The paper is organized as follows. In \cref{sec: 2}, we introduce the notation and recall some fundamental concepts. In \cref{sec: 3}, we provide an explicit expression for the operators that define the graph splitting methods and their fixed points. We specialize in \cref{sec: 4} the previous analysis when the operators $A_i$ are normal cones of closed linear subspaces, and compute the projection onto the set of fixed points of the operators defining the graph splitting methods. In addition, we derive the limit points of the sequences generated by the algorithms in this specific linear case and provide examples of particular splitting methods using the tools developed. Lastly, the concluding \cref{sec: concl} summarizes the major achievements.

\section{Preliminaries}\label{sec: 2}

Throughout, we assume that $\Hi$ and $\mathcal{K}$ are real Hilbert spaces with inner products $\langle\cdot,\cdot\rangle$ and induced norms~$\norm{\cdot}$. Elements in the Cartesian products $\Hi^n$ and $\R^n$ are marked with bold, e.g., $\bm{x}=(x_1,\ldots,x_n)\in\Hi^n$. The space $\Hi^n$ is endowed with a Hilbert space structure as follows: if $\bm{x},\bm{y}\in\Hi^n$, then $\langle\bm{x},\bm{y}\rangle:=\sum_{i=1}^n\langle x_i,y_i\rangle$. We will assume that vectors of any kind are column vectors. Also, whenever it is important to emphasize the arrangement of vectors, they are denoted with brackets and, otherwise, they are denoted with parentheses.

\subsection{Operator Theory}\label{sec: 2.1}

A \emph{set-valued operator} is a mapping $A:\Hi\rightrightarrows \Hi$ that assigns to each point $x\in \Hi$ a subset $A(x)$ of $\Hi$. If $T$ is a mapping such that $T(x)$ is a singleton for all $x\in \Hi$, then $T$ is said to be a \emph{single-valued operator} and we denote it by $T:\Hi\to \Hi$. In this case, we will also write $T(x)=y$ instead of $T(x)=\{y\}$.

Given a set-valued operator $A:\Hi\rightrightarrows \Hi$, the \emph{domain}, the \emph{range}, the \emph{graph}, the \emph{fixed points} and the \emph{zeros} of $A$ are, respectively,
\begin{align*}
\operatorname{dom}A&:=\{x\in \Hi: A(x)\neq\varnothing\}, & \operatorname{ran}A&:=\{u\in \Hi: u\in A(x)\text{ for some }x\in \Hi\},\\
\operatorname{gra}A&:=\{(x,u)\in \Hi\times \Hi: u\in A(x)\},& \operatorname{fix}A&:=\{x\in \Hi: x\in A(x)\},\\
\operatorname{zer}A&:=\{x\in \Hi: 0\in A(x)\}.
\end{align*}
The \emph{inverse} is the set-valued operator $A^{-1}:\Hi\rightrightarrows \Hi$ such that $x\in A^{-1}(u)\Leftrightarrow u\in A(x)$.

\begin{definition}\label{def: monotone}
    We say that $A:\Hi\rightrightarrows \Hi$ is \emph{monotone} if
    $$\langle x-y,u-v\rangle\geq0\quad \forall (x,u),(y,v)\in \operatorname{gra}A.$$
    Further, $A$ is \emph{maximally monotone} if for all $A':\Hi\rightrightarrows \Hi$ monotone, $\operatorname{gra}A\subseteq\operatorname{gra}A'\Rightarrow A=A'$.
\end{definition}

In the context of splitting algorithms, the resolvent operator is central. It is defined as follows.

\begin{definition}\label{def: resolvent}
    Let $A:\Hi\rightrightarrows \Hi$. The \emph{resolvent} of $A$ is
    \[J_A:=(\Id_\Hi+A)^{-1}.\]
\end{definition}

Maximally monotone operators naturally arise in convex analysis. For instance, the \emph{subdifferential} operator of a proper, convex and lower semicontinuous function, $\partial f$, is maximally monotone and its resolvent is the so-called \emph{proximal mapping} (see \cite[Theorem~20.25 and Example~23.3]{bauschke}). Similarly, if $C\subset\Hi$ is a nonempty, closed and convex set, the \emph{normal cone} operator, $N_C$, is maximally monotone and its resolvent coincides with the \emph{metric projection} onto $C$, denoted by $P_C$ (see \cite[Examples 20.26 and 23.4]{bauschke}).

The following classical result due to Minty characterizes maximal monotonicity in terms of some properties of the resolvent.

\begin{fact}[\cite{minty}]\label{lem: single-valued full domain resolvent}
    Let $A:\Hi\rightrightarrows \Hi$ be monotone. Then:

    \begin{enumerate}[label=(\roman*)]
        \item $J_A$ is single-valued;
        \item $A$ is maximally monotone if and only if $\operatorname{dom} J_A=\Hi$.
    \end{enumerate}
\end{fact}

Let us now turn our attention to single-valued operators. Let $T:\Hi\to\mathcal{K}$ be a continuous linear operator. We denote by $T^\ast$ the \emph{adjoint} of $T$, i.e., the continuous linear operator $T^\ast:\mathcal{K}\to \Hi$ such that $\langle Tx,y\rangle=\langle x,T^\ast y\rangle$ for all $x\in\Hi,y\in \mathcal{K}$.

\begin{definition}\label{def: self-adjoint orthogonal PSD}
    A continuous linear operator $T:\Hi\to \Hi$ is said to be \emph{self-adjoint} if $T=T^\ast$. In addition, $T$~is \emph{positive semidefinite} if $\langle Tx,x\rangle\geq0$ for all $x\in \Hi$.
\end{definition}

Most of the linear operators treated in this work have the form of a Kronecker product ${K}\otimes\Id_\Hi$, where $K\in\R^{m\times n}$ is a real matrix and $\Id_\Hi$ is the identity operator. To simplify, we will denote the continuous linear operator $\ot{K}:=K\otimes\Id_\Hi:\Hi^n\to\Hi^m$, which is defined for all $\bm{x}\in\Hi^n$ by
\begin{equation}\label{eq: KronMatrix}
\ot{K}\bm{x}:=\begin{bmatrix}
	K_{1,1}\Id_\mathcal{H} & \cdots & K_{1,n}\Id_\mathcal{H} \\
	 \vdots & \ddots & \vdots \\
	K_{m,1}\Id_\mathcal{H} & \cdots & K_{m,n}\Id_\mathcal{H}
\end{bmatrix}\begin{bmatrix}
	x_1 \\ \vdots \\ x_n
\end{bmatrix}=\begin{bmatrix}
	\sum_{j=1}^nK_{1,j}x_j \\ \vdots \\ \sum_{j=1}^nK_{m,j}x_j
\end{bmatrix}.
\end{equation}

Notice first that, denoting $I_n$ the identity matrix of order $n$, then $\ot{(I_n)}=\Id_{\Hi^n}$. On the other hand, recall that the Kronecker product respects matrix operations such as multiplication, transpose and inverse. With this notation, we obtain the equalities $\ot{(K_1)}\ot{(K_2)}=\ot{(K_1K_2)}$, $(\ot{K})^*=\ot{(K^*)}$ and, if $K$ is invertible, $(\ot{K})^{-1}=\ot{(K^{-1})}$.

\newpage

Likewise, given any vector $\bm{\alpha}\in\R^m$, the linear operator $\ot{\bm{\alpha}}:\Hi\to\Hi^m$ is defined for $x\in\Hi$ as
\begin{equation}\label{eq: KronVector}
\ot{\bm{\alpha}} x:=(\alpha_1 x,\ldots,\alpha_m x).
\end{equation}
Thus, $\ot{\bm{\alpha}^*}:\Hi^m\to\Hi$, where for all $\bm{v}\in\Hi^m$,
\[\ot{\bm{\alpha}^*}\bm{v}=\sum_{j=1}^m\alpha_jv_j.\]
We denote the \emph{diagonal set} as $\Delta_n:=\ot{(\bm{1}_n)}\Hi=\{(x,\ldots,x)\in\Hi^n\mid x\in\Hi\}$, where $\bm{1}_n\in\mathbb{R}^n$ is the vector whose components are all one.

Finally, we recall the notion of a generalized inverse for continuous linear operators with closed range. For some references on this topic, we suggest \cite{bauschke} and \cite{groetsch}.

\begin{definition}\label{def: Moore--Penrose inverse}
    Let $T:\Hi\to\mathcal{K}$ be a continuous linear operator with closed range. The \emph{Moore--Penrose inverse} of $T$ is the continuous linear operator $T^\dagger:\mathcal{K}\to \Hi$ such that $TT^\dagger=P_{\ran T}$ and $T^\dagger T=P_{\ran T^\dagger}$.
\end{definition}

Although there are many equivalent definitions for the Moore--Penrose inverse, the existence and uniqueness of $T^\dagger$ is always guaranteed (see, e.g., {\cite[Theorems 2.2.1 and 2.2.2]{groetsch}}).

\begin{fact}[{\cite[Proposition 3.30 and Exercise 3.13]{bauschke}}]\label{fact: Moore--Penrose}
    Let $T:\Hi\to\mathcal{K}$ be linear and continuous with closed range. Then, the following hold:
    \begin{enumerate}[label=(\roman*)]
        \item\label{fact: range} $\ran T^\dagger =\ran T^\ast$.
        \item\label{fact: kernel} $\ker T^\dagger=\ker T^\ast$.
    \end{enumerate}
\end{fact}

\begin{remark}\label{rem: Moore--Penrose inverse}
    If $T$ is injective, it can be shown (see {\cite[Example 3.29]{bauschke}}) that $T^\dagger=(T^\ast T)^{-1}T^\ast$, which, consequently, leads to the identity $T^\dagger T=\Id_\Hi$. Hence, if $K\in\R^{m\times n}$ is a matrix with full column rank, then $(\ot K)^\dagger=\ot{(K^\dagger)}$.
    A relevant example for this work is the case  of the linear mapping defined by~\cref{eq: KronVector} for a vector $\bm{\alpha}\in\R^m$.
    If $\bm{\alpha}\neq0$, we obtain
    \[\bm{\alpha}^\dagger=(\bm{\alpha}^*\bm{\alpha})^{-1}\bm{\alpha}^*=\frac{\bm{\alpha}^*}{\norm{\bm{\alpha}}^2}.\]
    On the other hand, if $\bm{\alpha}=0$, then by \cref{def: Moore--Penrose inverse} we get $0=\bm{\alpha}^\dagger\bm{\alpha}=P_{\ran \bm{\alpha}^\dagger}$, which is equivalent to say that $\ran \bm{\alpha}^\dagger=0$, i.e., $\bm{\alpha}^\dagger=0$.
\end{remark}

\subsection{Preconditioned proximal-point problems}\label{sec: 2.2}

A way to tackle problem \cref{eq: P} is by solving a \emph{preconditioned proximal point problem}, which is based on transforming \cref{eq: P} into a fixed-point problem by means of a positive semidefinite operator $M:\Hi\to\Hi$.
\begin{definition}\label{def: admissible preconditioner}
    Let $A:\Hi\rightrightarrows \Hi$ be a set-valued operator and $M:\Hi\to \Hi$ be a linear, continuous, self-adjoint and positive semidefinite operator. We say that $M$ is an \emph{admissible preconditioner} for $A$ if
    \begin{equation}\label{eq: JMA}
        J_{M^{-1}A}=(M+A)^{-1}M\text{ is single-valued and has full domain.}
    \end{equation}
\end{definition}
By doing some algebraic manipulations, one can easily check that
\[0\in A(x)\Leftrightarrow Mx \in A(x) + Mx \Leftrightarrow x\in (A+M)^{-1}Mx=J_{M^{-1}A}(x).\]
This can also be expressed by the following equality:
\begin{equation}\label{eq: zeros fixed}
    \zer A=\Fix J_{M^{-1}A}.
\end{equation}

Since an admissible preconditioner $M$ is self-adjoint and positive semidefinite, by~\cite[Proposition 2.3]{degenerate-ppp}, it can be split as $M=CC^\ast$ for some injective operator $C:\mathcal{K}\to\mathcal{H}$ and some Hilbert space $\mathcal{K}$. If $\operatorname{ran}M$ is closed, this factorization is called an \emph{onto decomposition} of $M$. Such factorization is unique up to orthogonal transformations, see~\cite[Proposition~2.2]{graph-drs}. As we subsequently detail, this allows  rewriting~\cref{eq: JMA} in terms of the resolvent of the \emph{parallel composition} $$C^\ast\rhd A:=(C^*A^{-1}C)^{-1}.$$
\begin{fact}[{\cite[Theorem 2.13]{degenerate-ppp}}]\label{lem: parallel resolvent}
    Let $A:\Hi\rightrightarrows \Hi$, let $M$ be an admissible preconditioner for $A$ and let $M=CC^\ast$ be an onto decomposition. Then $C^\ast \rhd A$ is maximally monotone and
    \begin{equation}\label{eq: JCA}
        J_{C^\ast\rhd A}=C^\ast(M+A)^{-1}C.
    \end{equation}
\end{fact}

Thanks to this fact, we can bring \cref{eq: JMA} back and rewrite it in terms of $J_{C^\ast \rhd A}$. Indeed, recalling that $J_{M^{-1}A}=(M+A)^{-1}M=(M+A)^{-1}CC^\ast$, this yields
\begin{equation*}\label{eq: JMA JCA}
    J_{C^*\rhd A}C^*=C^*J_{M^{-1} A}.
\end{equation*}
Not only this, but the following proposition ensures a bijective correspondence between them.

\begin{fact}[{\cite[Proposition 2.1]{bauschke-fixedpoints}}]\label{fact: Fix T bijection}
    The operator $C^*|_{\Fix J_{M^{-1}A}}$ is a bijection from
    $\Fix J_{M^{-1}A}$ to $\Fix J_{C^*\rhd A}$, with inverse
    $(M+A)^{-1}C|_{\Fix J_{C^*\rhd A}}$.
\end{fact}

This result, together with \cref{eq: zeros fixed}, reveals that computing the zeros of $A$ is equivalent to finding the fixed points of either $J_{M^{-1}A}$ or $J_{C^*\rhd A}$. Motivated by this relationship, Bredies et al.~\cite{degenerate-ppp} propose to generate two sequences for finding points in $\zer A$:
\begin{align}
    u^{k+1}&:=(1-\theta_k)u^k+\theta_kT(u^k)\in\Hi,\text{ with }T:=J_{M^{-1}A}, \label{eq: algorithm T}\\
    v^{k+1}&:=(1-\theta_k)v^k+\theta_k\widetilde{T}(v^k)\in\mathcal{K}, \text{ with }\widetilde{T}:=J_{C^*\rhd A}, \label{eq: algorithm T tilde}
\end{align}
for some starting point $u^0\in\Hi$ and $v^0:=C^*u^0$, where $(\theta_k)_{k\in\N}$ is a sequence of relaxation parameters in~$[0,2]$ satisfying $\sum_{k\in\N}\theta_k(2-\theta_k)=+\infty$. The algorithm defined by \cref{eq: algorithm T} is called the \emph{preconditioned proximal point method}, while \cref{eq: algorithm T tilde} is the \emph{reduced preconditioned proximal point method}.

Before presenting the main results obtained in~\cite{degenerate-ppp} and~\cite{bauschke-fixedpoints} for the sequences generated by the preconditioned proximal point methods, we need the following notion of projection, introduced in~\cite[Definition 3.2]{bauschke-fixedpoints}.
\begin{definition}\label{def: M-proj}
The $M$\emph{-projection} of an element $u\in\Hi$ onto some subset $S\subseteq\Hi$, denoted as $P^M_S(u)$, is the projection of $u$ onto $S$ with the seminorm $\norm{\cdot}_M$ given by
$$\norm{x}_M:=\sqrt{\langle x,Mx\rangle}=\norm{C^*x}, \text{ for } x\in\Hi.$$
\end{definition}
In general, the $M$-projection is set-valued. This is not the case for $\Fix T$.
\begin{fact}[{\cite[Theorem 3.1]{bauschke-fixedpoints}}]\label{fact: projections relations}
The $M$-projection $P^M_{\Fix T}:\Hi\to\Fix T$ is single-valued. Further, the following holds
\begin{equation*}
    P^M_{\Fix T}=(M+A)^{-1}CP_{\Fix \widetilde{T}}C^*\quad\text{and}\quad
    C^*P^M_{\Fix T}=P_{\Fix \widetilde{T}}C^*.
\end{equation*}
\end{fact}

We can finally summarize the information about the limit points of the sequences $(u^k)_{k\in\N}$ and $(v^k)_{k\in\N}$ defined by the preconditioned proximal point methods.
\begin{fact}\label{fact: proj}
    Let $A: \Hi\tto\Hi$ be a maximally monotone operator with $\zer A \neq \emptyset$ and $M$ be an admissible preconditioner such that $(M+A)^{-1}$ is Lipschitz. Let $(u^k)_{k\in\N}$ and $(v^k)_{k\in\N}$ be the sequences generated by~\cref{eq: algorithm T} and~\cref{eq: algorithm T tilde} for some starting point $u^0\in\Hi$ and $v^0:=C^*u^0$.
    Then the following assertions hold:

    \begin{enumerate}[label=(\roman*)]
    \item \label{fact: proj_i} The sequence $(T(u^k))_{k\in\N}$ weakly converges to some point $\overline{u}\in\Fix T$. Furthermore, if $0<\inf\theta_k\leq\sup\theta_k<2$, then $(u^k)_{k\in\N}$ weakly converges to $\overline{u}$ as well.
    \item \label{fact: proj_ii} The sequence $(v^k)_{k\in\N}$ weakly converges to some point $\overline{v}\in\Fix \widetilde{T}$.
    \item \label{fact: proj_iii} It holds $v^k=C^*u^k$ and $T(u^k)=(M+A)^{-1}Cv^k$, for all $k\in\N$. Moreover, $\overline{v}=C^*\overline{u}$ and $\overline{u}=(M+A)^{-1}C\overline{v}$.
    \item \label{fact: proj_iv} If $\overline{v}=P_{\Fix \widetilde{T}}(v^0)$, then $\overline{u}=P^M_{\Fix T}(u^0)$.
    \item \label{fact: proj_v} If $\Fix \widetilde{T}$ is an affine subspace of $\mathcal{K}$, then $\overline{v}=P_{\Fix \widetilde{T}}(v^0)$.
    \item \label{fact: proj_vi} If $A$ is a linear relation (i.e., its graph
is a linear subspace of $\Hi\times\Hi$) and $(\theta_k)_{k\in\N}$ is identical to some constant $\theta\in{]0,2[}$, then $\Fix \widetilde{T}$ is a linear subspace and the convergence in \ref{fact: proj_i} and \ref{fact: proj_ii} is strong.
    \end{enumerate}
\end{fact}
\begin{proof}
\ref{fact: proj_i}: See~\cite[Corollary 2.10]{degenerate-ppp}. \ref{fact: proj_ii}: See~\cite[Corollary 2.15]{degenerate-ppp}. \ref{fact: proj_iii}: This was proved in~\cite{degenerate-ppp}, for more details, see~\cite[Fact~1.1]{bauschke-fixedpoints}. \ref{fact: proj_iv}: See~\cite[Theorem 3.6]{bauschke-fixedpoints}. \ref{fact: proj_v}: See~\cite[Theorem 3.4]{bauschke-fixedpoints}. \ref{fact: proj_vi}: See~\cite[Theorem 4.2 and Remark 3.5]{bauschke-fixedpoints}.
\end{proof}

\subsection{Graph Theory}\label{sec: 2.3}

We say that $G=(\mathcal{N},\mathcal{E})$ is a \emph{directed graph} if $\mathcal{N}$ is a finite set and $\mathcal{E}\subseteq\mathcal{N}\times\mathcal{N}$. The elements of $\mathcal{N}$ are referred as \emph{nodes}, while the elements of $\mathcal{E}$ are called \emph{edges}. The \emph{order} of the graph is the number of nodes~$|\mathcal{N}|$. Further, we say that $G'=(\mathcal{N},\mathcal{E}')$ is a \emph{spanning subgraph} of $G$ if $\mathcal{E}'\subseteq\mathcal{E}$.

Two nodes $i,j\in\mathcal{N}$ are \emph{adjacent} if $(i,j)\in\mathcal{E}$ or $(j,i)\in\mathcal{E}$. The \emph{in-degree}, \emph{out-degree} and \emph{degree} are defined respectively as  $d_i^{\rm in}:=|\{j\in\mathcal{N}:(j,i)\in\mathcal{E}\}|$, $d_i^{\rm out}:=|\{j\in\mathcal{N}:(i,j)\in\mathcal{E}\}|$ and
\begin{equation}\label{eq: degree}
d_i:=d_i^{\rm in}+d_i^{\rm out}.
\end{equation}

A \emph{path} in $G$ is a finite sequence of distinct nodes $(v_1,\ldots,v_r)$ with $r\geq2$ such that $v_k$ and $v_{k+1}$ are adjacent for all $k=1,\ldots,r-1$. The nodes $v_1$ and $v_r$ are called \emph{endpoints}. A graph $G$ is said to be \emph{connected} if for all distinct nodes $i,j\in\mathcal{N}$, there exists a path with $i$ and $j$ as endpoints. Further, $G$ is a \emph{tree} if it is connected and $|\mathcal{E}|=n-1$.

In this work all graphs are assumed to satisfy the following structure, which is suitable for defining splitting algorithms.

\begin{definition}\label{def: algorithmic graph}
    We say that $G=(\mathcal{N},\mathcal{E})$ is an \emph{algorithmic graph} if
    \begin{enumerate}[label=(\roman*)]
        \item $\mathcal{N}=\{1,\ldots,n\}$ with $n\geq2$,
        \item $(i,j)\in\mathcal{E}\Rightarrow i<j$, and
        \item $G$ is connected.
    \end{enumerate}
\end{definition}

Let us present some examples of algorithmic graphs.

\begin{example}
    Particular instances of algorithmic graphs $G=(\mathcal{N},\mathcal{E})$ of order $n$ include the following:

    \begin{enumerate}[label=(\roman*)]
        \item the \emph{complete} graph, in which every pair of distinct nodes are adjacent;
        \item the \emph{sequential} graph, in which $(1,\ldots,n)$ is a path and $G$ is a tree, i.e., $|\mathcal{E}|=n-1$;
        \item the \emph{ring} graph, in which $(1,\ldots,n)$ is a path, $(1,n)\in\mathcal{E}$ and $|\mathcal{E}|=n$;
        \item the \emph{star} graph with \emph{center} $i\in\mathcal{N}$, in which $d_i=n-1$ and $d_j=1$ for $j\neq i$;
        \item the \emph{parallel up} graph, which is the star graph whose center is the first node $1$;
        \item the \emph{parallel down} graph, which is the star graph whose center is the last node $n$.
    \end{enumerate}
\end{example}

Notice that, since the number of nodes of a graph is finite, so is the number of edges. Hence, we can also index $\mathcal{E}$ with natural numbers, which is useful for defining the following matrices.

\begin{definition}\label{def: incidence matrix}
    Let $G=(\mathcal{N},\mathcal{E})$ be an algorithmic graph of order $n$.
    \begin{enumerate}[label=(\roman*)]
    \item The \emph{incidence matrix} of $G$ is the matrix $\operatorname{Inc}(G)\in\mathbb{R}^{n\times |\mathcal{E}|}$ defined componentwise as
    \[\operatorname{Inc}(G)_{i,e}:=\begin{cases}
        1 & \text{if the edge }e\text{ leaves the node }i,\\
        -1 & \text{if the edge }e\text{ enters the node }i,\\
        0 & \text{otherwise}.
    \end{cases}\]
    \item The \emph{Laplacian matrix} is $\operatorname{Lap}(G):=\operatorname{Inc}(G)\operatorname{Inc}(G)^\ast\in\mathbb{R}^{n\times n}$, componentwisely described as
    \[\operatorname{Lap}(G)_{i,j}:=\begin{cases}
        d_i & \text{if }i=j,\\
        -1 & \text{if }i\text{ and }j\text{ are adjacent}\\
        0 & \text{otherwise.}
    \end{cases}\]
    \end{enumerate}
\end{definition}

\begin{fact}\label{fact: Zdecom}
Let $G$ be an algorithmic graph of order $n$. Let $\lambda_1,\ldots,\lambda_{n-1}$ be the positive eigenvalues of $\operatorname{Lap}(G)$ and $v_1,\ldots,v_{n-1}$ be some associated orthonormal eigenvectors. Take
\begin{equation}\label{eq: Zeig}
Z:=[v_1~v_2~\cdots~v_{n-1}]\operatorname{Diag}\left(\sqrt{\lambda_1},\ldots,\sqrt{\lambda_{n-1}}\right)\in\mathbb{R}^{n\times (n-1)}.
\end{equation}
Then, $Z$ provides an onto decomposition of $\operatorname{Lap}(G)$, in the sense that
\begin{equation}\label{eq:Zdecom}
\operatorname{Lap}(G)=ZZ^\ast,
\end{equation}
which implies $\rk Z=n-1$, $\ker Z=\{\mathbf{0}_{n-1}\}$, and $\ker Z^\ast=\operatorname{span}\{\bm{1}_n\}$. Further, if $Z$ is chosen as in~\cref{eq: Zeig}, then
\begin{equation}\label{eq: Zdagger}
Z^\dagger=\operatorname{Diag}\left(\lambda_1^{-1/2},\ldots,\lambda_{n-1}^{-1/2}\right)[v_1~v_2~\cdots~v_n]^*.
\end{equation}
In particular, an onto decomposition $Z$ satisfying~\cref{eq:Zdecom} can be obtained as follows:

\begin{enumerate}[label=(\roman*)]
\item \label{fact: Zdecom_i} If $G$ is a tree, then we can take $Z=\Inc(G)$.
\item \label{fact: Zdecom_iii} If $G$ is a \emph{circulant graph} (i.e., $\operatorname{Lap}(G)$ is a circulant matrix) and $(c_1,\ldots,c_n)\in\R^n$ are the elements in the first row of $\operatorname{Lap}(G)$, then
$$\lambda_j=\sum_{k=1}^n c_k \cos\left(\frac{2\pi j(k-1)}{n}\right),\quad j=1,\ldots,n-1,$$
and one can take as their associated orthonormal eigenvectors
$$v_j=\sqrt{\frac{2}{n}}{\left(\cos\left(\frac{2\pi (i-1)j}{n}-\frac{\pi}{4}\right)\right)}_{i=1,\ldots,n}\in\R^n,\quad j=1,\ldots,n-1.$$
Therefore, $Z$ in~\cref{eq: Zeig} and $Z^\dagger$ in~\cref{eq: Zdagger} are given componentwise by
      $$Z_{i,j}=\cos\left(\frac{2\pi (i-1)j}{n}-\frac{\pi}{4}\right)\sqrt{\frac{2}{n}\sum_{k=1}^n c_k \cos\left(\frac{2\pi j(k-1)}{n}\right)}\quad\text{and}\quad Z_{j,i}^\dagger=\frac{Z_{i,j}}{\lambda_j}.$$
\item \label{fact: Zdecom_ii} If $G$ is complete, then~\ref{fact: Zdecom_iii} also applies. In addition, another sparser decomposition is given by
\begin{equation}\label{eq: Z complete}
    Z_{i,j}:=\begin{cases}
    (n-i)t_i & \text{if }i=j,\\
    -t_j &\text{if } i>j\\
    0&\text{otherwise}.
    \end{cases}
\end{equation}
where $t_j:=\sqrt{\frac{n}{(n-j)(n-j+1)}}$.

\end{enumerate}
\end{fact}
\begin{proof}
For the first part, see~\cite[Proposition 2.16]{graph-fb}. The expression of $Z^\dagger$ follows from \cref{rem: Moore--Penrose inverse}.
\ref{fact: Zdecom_i}: See~\cite[Remark~2.17]{graph-fb}.
\ref{fact: Zdecom_iii}: This is a consequence of \cref{eq: Zeig}, \cite[Theorem~14.5.10]{Garcia2023} and the trigonometric identity $\cos(x)+\sin(x)=\sqrt{2}\cos\left(x-\frac{\pi}{4}\right)$. \ref{fact: Zdecom_ii}: See~\cite[Proposition A.2]{graph-fb}.
\end{proof}

\begin{proposition}\label{prop: ker Z ran Z}
Let $G$ be an algorithmic graph of order $n$ and let $Z\in\R^{n\times(n-1)}$ be an onto decomposition of its Laplacian satisfying~\cref{eq:Zdecom}. Then, the linear and continuous operator $\ot{Z}:\Hi^{n-1}\to\Hi^n$ given by~\cref{eq: KronMatrix} satisfies
\begin{equation}
\ran \ot{Z}=\Delta_n^\perp=\left\{\bm{w}\in\Hi^n\mid w_1+\cdots+w_n=0\right\}.\label{eq:Zran}
\end{equation}
Consequently, $\ran \ot{Z}$ is closed and it holds
\begin{equation*}
\ker \ot{Z^*} =\Delta_n=\left\{\ot{(\bm{1}_n)} u \mid u\in\Hi\right\}.
\end{equation*}
\end{proposition}

\begin{proof}
We know from \Cref{fact: Zdecom} that $\ran(Z)=\ker(Z^*)^\perp=\operatorname{span}\{\bm{1}_n\}^\perp$. Hence, using \cite[Equation (1.11) and Exercise 2(b) pp. 35]{greub}, we compute
\begin{align*}
\ran(\ot Z)&=\ran(Z)\otimes\ran(\Id_\Hi)=\operatorname{span}\{\bm{1}_n\}^\perp\otimes\Hi\\ &=\left(\operatorname{span}\{\bm{1}_n\}\otimes\Hi\right)^\perp=\left(\ran(\bm{1}_n)\otimes\ran(\Id_\Hi)\right)^\perp= \ran(\ot{(\bm{1}_n)})^\perp=\Delta_n^\perp,
\end{align*}
which proves the relation in \cref{eq:Zran}. The remaining assertions are direct consequences of the closedness of $\Delta_n^\perp$ and \cite[Fact~2.25(v)]{bauschke}.
\end{proof}

\section{On the graph splitting framework} \label{sec: 3}

In this section, we provide an explicit expression of the operators given in \cref{eq: JMA,eq: JCA} when they are considered in the context of graph splitting methods. This will allow us to study their sets of fixed points, which will be instrumental in our subsequent analysis of the limit points of the algorithms. The design of both operators relies on the following assumptions.

\begin{assumption}\label{as: graphs} Let $A_1,\ldots,A_n:\Hi\tto\Hi$ be some given maximally monotone operators with $\zer(\sum_{i=1}^nA_i)\neq\emptyset$. Let $(G, G')$ be a pair of graphs of order $n\geq 2$ satisfying the following:\smallskip

    \begin{enumerate}[leftmargin=1.8cm,label=(AG\arabic*)]
        \item\label{AG1} $G=(\mathcal{N},\mathcal{E})$ is an algorithmic graph.
        \item\label{AG2} $G'=(\mathcal{N},\mathcal{E}')$ is a connected spanning subgraph of $G$.
    \end{enumerate}
\end{assumption}
Throughout, we suppose that \cref{as: graphs} holds. We denote the degrees of the graph $G$ by $d_i>0$, given by~\cref{eq: degree}, and those of the subgraph $G'$ by $d_i'>0$, for $i\in\mathcal{N}$.

\subsection{Description of the operators and the algorithms}

In what follows, we make use of the operators $\mathcal{A}$ and $\mathcal{M}$, introduced in \cite{graph-drs}.

\paragraph{The preconditioner $\mathcal{M}$:}

Let $Z\in\mathbb{R}^{n\times (n-1)}$ be an onto decomposition of $\operatorname{Lap}(G')$. We define the preconditioner $\mathcal{M}:\Hi^{2n-1}\to \Hi^{2n-1}$ and its onto decomposition $\mathcal{M}=\mathcal{C}\mathcal{C}^\ast$, with $\mathcal{C}:\Hi^{n-1}\to \Hi^{2n-1}$, as
\begin{equation}\label{eq: M}
\mathcal{M}:=\ot{\begin{bmatrix}
    \Lap(G') &  Z \\  Z^\ast & I_{n-1}
\end{bmatrix}}\quad\text{and}\quad \mathcal{C}:=\ot{\begin{bmatrix}
    Z \\ I_{n-1}
\end{bmatrix}}.
\end{equation}

\paragraph{The operator $\mathcal{A}$:}

We define $\mathcal{A}:\Hi^{2n-1}\rightrightarrows \Hi^{2n-1}$ as
\begin{equation*}
    \mathcal{A}:=\begin{bmatrix}
        \mathcal{A}_D+\ot{P(G)}-\ot{\Lap(G')} & -\ot{Z} \\  \ot{Z^\ast} & \ot{\left(\bm{0}_{(n-1)\times(n-1)}\right)}
\end{bmatrix},
\end{equation*}
where $\mathcal{A}_D:=\Diag(A_1,\ldots,A_n)$, that is,
$$\mathcal{A}_D(\bm{x})=A_1(x_1)\times\cdots\times A_n(x_n),\quad\forall \bm{x}=(x_1,\ldots,x_n)\in\Hi^n,$$
and $P(G)\in\R^{n\times n}$ is given componentwise by
$$
P(G)_{i,j}=\begin{cases}
d_i&\text{if }i=j,\\
-2 &\text{if }(j,i)\in\E,\\
0 &\text{otherwise}.
\end{cases}$$
\begin{remark}\label{rem: A}
Although the definition of the first block element of $\mathcal{A}$ appears to be different from that of~\cite{graph-drs,graph-fb}, some simple arithmetic operations show that this operator is the same. Unlike these works, we consider a stepsize $\tau$ equal to $1$ for simplicity (the general case with $\tau>0$ can be subsumed to this case by considering the operators $\tau A_i$). Moreover, by~\cite[Theorem~3.1]{graph-drs}, the operator $\mathcal{A}$ is maximally monotone and one has $(\bm{x},\bm{v})\in\zer\mathcal{A}$ if and only if $x:=x_1=\cdots=x_n$ solves~\cref{eq: P}.
\end{remark}

\paragraph{The operators $\mathcal{T}$ and $\widetilde{\mathcal{T}}$:} We consider $\mathcal{T}:\Hi^{2n-1}\to\Hi^{2n-1}$ and $\widetilde{\mathcal{T}}:\Hi^{n-1}\to\Hi^{n-1}$ given by
\begin{align}
    \mathcal{T}&:=J_{\mathcal{M}^{-1}\mathcal{A}}=(\mathcal{M}+\mathcal{A})^{-1}\mathcal{M},\label{eq: T}\\
    \widetilde{\mathcal{T}}&:=J_{\mathcal{C}^*\rhd\mathcal{A}}=\mathcal{C}^*(\mathcal{M}+\mathcal{A})^{-1}\mathcal{C}\label{eq: T tilde}.
\end{align}

We begin our study by providing an explicit description of the operator $(\mathcal{M}+\mathcal{A})^{-1}$, which will be useful for deriving explicit expressions for the operators $\mathcal{T}$ and $\widetilde{\mathcal{T}}$. The first assertion in the next result was observed in the proof of~\cite[Theorem~3.2]{graph-drs}.

\begin{lemma}\label{lem:invMA}
The mapping $(\mathcal{M}+\mathcal{A})^{-1}:\Hi^{2n-1}\to\Hi^{2n-1}$ is single-valued and Lipschitz. Specifically, for any $\bm{w}\in\Hi^n$ and $\bm{v}\in\Hi^{n-1}$ one has
$$(\mathcal{M}+\mathcal{A})^{-1}\begin{bmatrix} \bm{w} \\ \bm{v}\end{bmatrix}=\begin{bmatrix}
\bm{x} \\ \bm{v}-2\ot{Z^*}\bm{x}\end{bmatrix},$$
where $\bm{x}\in\Hi^{n}$  can be explicitly computed as
\begin{equation}\label{eq: M+A inverse}
            x_i = J_{\frac{1}{d_i}A_i}\bigg(\frac{w_i}{d_i} +\frac{2}{d_i}\sum_{(h,i)\in\mathcal{E}}x_h\bigg),\quad i= 1,\ldots,n.
\end{equation}
\end{lemma}

\begin{proof}
For any $\bm{x},\bm{w}\in\Hi^n$ and $\bm{v},\bm{y}\in\Hi^{n-1}$, one has
\[\begin{bmatrix}
    \bm{x} \\ \bm{y}
\end{bmatrix}\in(\mathcal{M}+\mathcal{A})^{-1}\begin{bmatrix}
    \bm{w} \\ \bm{v}
\end{bmatrix}\Leftrightarrow\begin{bmatrix}
    \bm{w} \\ \bm{v}
\end{bmatrix}\in \begin{bmatrix}
    \mathcal{A}_D+\ot{P(G)} & \ot{(\bm{0}_{n\times(n-1)})} \\
    2 \ot{Z^\ast} & \Id_{\Hi^{n-1}}
    \end{bmatrix}\begin{bmatrix}
        \bm{x} \\ \bm{y}
    \end{bmatrix}=\begin{bmatrix}
        (\mathcal{A}_D+\ot{P(G)})(\bm{x}) \\ 2 \ot{Z^\ast}\bm{x}+\bm{y}
    \end{bmatrix}.\]
Finally, by the definition of $P(G)$, the $i$-th equation of the previous system is
\[w_i\in\big(\left(\mathcal{A}_D+\ot{P(G)}\right)(\bm{x})\big)_i= A_i (x_i) + d_i x_i - 2 \sum_{(h,i)\in\mathcal{E}}x_h=d_i\left(\Id_\Hi + \frac{1}{d_i} A_i\right) (x_i) - 2 \sum_{(h,i)\in\mathcal{E}}x_h.\]
Solving for $x_i$ in the previous equation gives the general formula~\cref{eq: M+A inverse}. The assertion regarding the explicit computation of $\bm{x}$ is a consequence of~\ref{AG1} in \cref{as: graphs} and~\cref{eq: M+A inverse}, since $(h,i)\in\mathcal{E}$ implies $h<i$.
\end{proof}

\begin{proposition}[Operator $\mathcal{T}$]\label{prop: operator T}
    Let $\mathcal{T}:\Hi^{2n-1}\to\Hi^{2n-1}$ be the operator defined in~\cref{eq: T}. Then, for any $\bm{w}\in\Hi^n$ and $\bm{v}\in\Hi^{n-1}$ one has
    $$        \mathcal{T}\begin{bmatrix}
            \bm{w} \\ \bm{v}
        \end{bmatrix}=\begin{bmatrix}
            \bm{x} \\ \bm{v}+ \ot{Z^\ast}(\bm{w}-2\bm{x})
        \end{bmatrix},$$
    where $\bm{x}\in\Hi^n$ can be explicitly computed as
    \begin{equation*}
        x_i= J_{\frac{1}{d_i}A_i}\bigg(\frac{d'_i}{d_i} w_i - \frac{1}{d_i}\sum_{(i,h)\in\mathcal{E}'}w_h - \frac{1}{d_i}\sum_{(h,i)\in\mathcal{E}'}w_h+ \frac{1}{d_i}( \ot{Z}\bm{v})_i+\frac{2}{d_i}\sum_{(h,i)\in\mathcal{E}}x_h\bigg),\quad i=1,\ldots,n.
    \end{equation*}
\end{proposition}

\begin{proof}
Combining \cref{lem:invMA} with \cref{eq: M} provides the description of the operator.
\end{proof}

\begin{proposition}[Operator $\widetilde{\mathcal{T}}$]\label{prop: operator T tilde}
    Let $\widetilde{\mathcal{T}}:\Hi^{n-1}\to\Hi^{n-1}$ be the operator defined in~\cref{eq: T tilde}. Then, for any $\bm{v}\in\Hi^{n-1}$, one has
    \begin{equation}\label{eq: operator T tilde}
    \widetilde{\mathcal{T}}(\bm{v})=\bm{v}-\ot{Z^\ast}\bm{x},\text{ where }x_i = J_{\frac{1}{d_i}A_i}\bigg(\frac{1}{d_i}( \ot{Z}\bm{v})_i + \frac{2}{d_i}\sum_{(h,i)\in\mathcal{E}}x_h\bigg),\quad i= 1,\ldots,n.
    \end{equation}
\end{proposition}

\begin{proof}
We have $\widetilde{\mathcal{T}}(\bm{v})=\begin{bmatrix} \ot{Z^\ast} & \Id_{\Hi^{n-1}}\end{bmatrix}(\mathcal{M}+\mathcal{A})^{-1}\begin{bsmallmatrix} \ot Z\bm{v} \\ \bm{v}\end{bsmallmatrix}$. Then, using \cref{lem:invMA}, we deduce $\widetilde{\mathcal{T}}(\bm{v})=\ot Z^*\bm{x}+\bm{v}-2\ot Z^*\bm{x}$, with $x_i$ given by~\cref{eq: operator T tilde},  as claimed.
\end{proof}

Using the expression provided by \cref{prop: operator T} and \cref{eq: algorithm T}, we obtain the iterative scheme presented in \cref{alg:T}, whose convergence is guaranteed by \cref{fact: proj}.

\begin{algorithm}[ht!]
    \caption{Graph splitting method induced by ${\mathcal{T}}$}\label{alg:T}
    \begin{algorithmic}[1]
        \State \textbf{let:} $w_1^0,\ldots,w_{n}^0,v_1^0,\ldots,v_{n-1}^0\in\mathcal{\Hi}$ and $\theta_k \in {[0, 2]}$ with $\sum_{k=0}^{\infty}\theta_k\left( 2 -\theta_k\right) = +\infty$.
        \For {$k=0,1,2,\ldots$}
        \State {\small$\displaystyle x_i^{k+1}= J_{\frac{1}{d_i}A_i}\bigg(\frac{d'_i}{d_i} w_i^{k} - \frac{1}{d_i}\sum_{(i,h)\in\mathcal{E}'}w_h^{k} - \frac{1}{d_i}\sum_{(h,i)\in\mathcal{E}'}w_h^{k}+ \frac{1}{d_i}\sum_{j=1}^{n-1}Z_{i,j}v^k_j+\frac{2}{d_i}\sum_{(h,i)\in\mathcal{E}}x_h^{k+1}\bigg)$}, \hfill{$i= 1,\ldots,n$}
        \State {$w_i^{k+1}=(1-\theta_k)w_i^k+\theta_k x_i^{k+1}$}, \hfill $i=1,\ldots,n$
        \State {$v_i^{k+1}=v_i^k+\theta_k\sum_{j=1}^n Z_{j,i}\left(w_j^{k}-2x_j^{k+1}\right)$}, \hfill{$i=1,\ldots, n-1$}
        \EndFor
    \end{algorithmic}
\end{algorithm}

One can observe in \cref{prop: operator T tilde} that $\widetilde{\mathcal{T}}$ is the exact operator used for defining~\cite[Algorithm~3.1]{graph-drs}, which is presented below in \cref{alg:Ttilde}. Under the graph splitting framework, $\widetilde{\mathcal{T}}$ can always be used because the preconditioner $\mathcal{M}$ defined in~\cref{eq: M} has an onto decomposition, since the Laplacian matrix $\operatorname{Lap}(G')$ always admits one (recall \cref{fact: Zdecom}). As shown in~\cite{graph-drs,graph-fb}, many algorithms can be described in terms of graphs.

\begin{algorithm}[ht!]
    \caption{Graph splitting method induced by $\widetilde{\mathcal{T}}$}\label{alg:Ttilde}
    \begin{algorithmic}[1]
        \State \textbf{let:} $v_1^0,\ldots,v_{n-1}^0\in\mathcal{\Hi}$ and $\theta_k \in {[0, 2]}$ with $\sum_{k=0}^{\infty}\theta_k\left( 2 -\theta_k\right) = +\infty$.
        \For {$k=0,1,2,\ldots$}
        \State {$x_i^{k+1}=J_{\frac{1}{d_i}A_i}\left(\frac{2}{d_i}\sum_{(h,i)\in\mathcal{E}}x_h^{k+1}+\frac{1}{d_i}\sum_{j=1}^{n-1}Z_{i,j}v^k_j\right)$}, \hfill{$i= 1,\ldots,n$}
        \State {$v_i^{k+1}=v_i^k-\theta_k\sum_{j=1}^n Z_{j,i}x_j^{k+1}$}, \hfill{$i=1,\ldots,n-1$}
        \EndFor
    \end{algorithmic}
\end{algorithm}

Note that \cref{alg:T} generates two sequences of governing variables, namely $(\bm{w}^k)_{k\in\N}$ in $\Hi^n$ and $(\bm{v}^k)_{k\in\N}$ in $\Hi^{n-1}$, while \cref{alg:Ttilde} only produces one governing sequence $(\bm{v}^k)_{k\in\N}$ in $\Hi^{n-1}$. In the particular case where $\theta_k=1$, we observe that the governing sequence $(\bm{w}^k)_{k\in\N}$ generated by \cref{alg:T} coincides with the shadow sequence $(\bm{x}^k)_{k\in\N}$. Even so, the ambient space of Algorithm 1 is always $\Hi^{2n-1}$, as the computation of the shadow variables in line~3 requires the storage of both $\bm{w}^k$ and $\bm{v}^k$.

\subsection{Computing fixed points}

According to~\cref{eq: zeros fixed}, \cref{rem: A} and \cref{fact: Fix T bijection}, the solutions to problem~\cref{eq: P} are characterized by the sets of fixed points of the operators $\mathcal{T}$ and $\widetilde{\mathcal{T}}$. In order to describe these sets, we define the operator $\mathcal{A}_n:=\mathcal{A}_D\ot{(\bm{1}_n)}$, i.e.,
\[\mathcal{A}_n(x)=A_1(x)\times\cdots\times A_n(x),\quad \forall x\in\Hi.\]
It will also be useful to define the following vector.

\begin{definition} Given an algorithmic graph $G$, we define the \emph{degree balance} as the vector $\bm{\delta}=(\delta_1,\ldots,\delta_n)\in\R^n$ given by
\begin{equation*}
    \delta_i:=d_i^\text{out}-d_i^\text{in}, \quad \text{for all } i=1,\ldots,n,
\end{equation*}
where $d_i^\text{in}$ and $d_i^\text{out}$ are the in- and out-degrees of node $i$.
\end{definition}
Notice that $\delta_i=d_i-2d_i^\text{in}$. Also, one has $\delta_1=d_1>0$ and $\delta_n=-d_n<0$, since $d_1^\text{in}=d_n^\text{out}=0$. Not only that, but also $\sum_{i=1}^n\delta_i=\sum_{i=1}^n d_i^\text{out}-\sum_{i=1}^n d_i^\text{in}=0$, so $\ran\ot{\bm{\delta}}=\ot{\bm{\delta}}(\Hi)\subseteq\Delta_n^\perp$.

\begin{lemma}[Fixed points]\label{lem: Fix T general}
    The fixed points of the operators $\mathcal{T}$ and $\widetilde{\mathcal{T}}$ defined in~\cref{eq: T} and~\cref{eq: T tilde} are given by:
    \begin{align*}
        \Fix \mathcal{T}&=\left\{(\ot{(\bm{1}_n)}x,\bm{v})\in\Hi^{2n-1}\mid \ot Z\bm{v} \in \mathcal{A}_n (x) + \ot{\bm{\delta}}x\right\},\\
        \Fix \widetilde{\mathcal{T}}&=\left\{\bm{v}\in\Hi^{n-1}\mid \ot Z\bm{v} \in \mathcal{A}_n(x) + \ot{\bm{\delta}}x, \text{ with }x\in\Hi\right\},
    \end{align*}
    where $\bm{\delta}\in\R^n$ is the degree balance of $G$. What is more, for all $\bm{v}\in\Hi^{n-1}$, the inclusion $\ot Z\bm{v} \in \mathcal{A}_n(x) + \ot{\bm{\delta}}x$ has a unique solution given by
    \begin{equation}\label{eq: x from v}
      x_{\bm{v}}=J_{\frac{1}{d_1}A_1}\left(\frac{(\ot Z\bm{v})_1}{d_1}\right),
    \end{equation}
    which implies $x_{\bm{v}}\in\zer(\sum_{i=1}^n A_i)$.
\end{lemma}

\begin{proof}
Let us first prove the expression for $\Fix {\mathcal{T}}$. By \cref{prop: operator T}, $(\bm{w},\bm{v})\in\Fix {\mathcal{T}}\Leftrightarrow\bm{w}=\bm{x}$ and $\bm{v}=\bm{v}-\ot Z^*\bm{w}\Leftrightarrow \bm{0}_n=\ot Z^*\bm{x}$, where $\bm{x}=(x_1,\ldots,x_n)$ is given by
\begin{equation*}
        x_i= J_{\frac{1}{d_i}A_i}\bigg(\frac{d'_i}{d_i} x_i - \frac{1}{d_i}\sum_{(i,h)\in\mathcal{E}'}x_h - \frac{1}{d_i}\sum_{(h,i)\in\mathcal{E}'}x_h+ \frac{1}{d_i}(\ot Z\bm{v})_i+\frac{2}{d_i}\sum_{(h,i)\in\mathcal{E}}x_h\bigg),\quad i= 1,\ldots,n.
    \end{equation*}
Since $\ker \ot{Z}^*=\Delta_n$ by
    \cref{prop: ker Z ran Z}, this is equivalent to $x_1=\cdots=x_n=:x$. That is to say, for all $i=1,\ldots,n$, we have that
    \begin{equation*}
        x=J_{\frac{1}{d_i}A_i}\left(\frac{d'_i-{(d'_i)}^\text{out}-{(d'_i)}^\text{in}}{d_i}x+ \frac{1}{d_i}(\ot Z\bm{v})_i+\frac{2d_i^\text{in}}{d_i}x\right)=
        J_{\frac{1}{d_i}A_i}\left(\frac{(\ot Z\bm{v})_i+2d_i^\text{in}x}{d_i}\right),
    \end{equation*}
    which is equivalent to
    $$\frac{(\ot Z\bm{v})_i+2d_i^\text{in}x}{d_i}\in\frac{1}{d_i}A_i(x)+x.$$
    Solving for $(\ot Z\bm{v})_i$, we get that $(\ot Z\bm{v})_i\in A_i(x)+\delta_i x$ for all $i=1,\ldots,n$, i.e., $\ot Z\bm{v}\in \mathcal{A}_n(x)+\ot{\bm{\delta}}x$.

    Now, from \cref{fact: Fix T bijection}, it follows that

    \begin{equation*}
    \Fix \widetilde{\mathcal{T}}= \left\{ \mathcal{C}^\ast\begin{bmatrix}
    \ot{(\bm{1}_n)}x\\ \bm{v}
    \end{bmatrix}= \bm{v} \;\biggl\lvert\;  \ot Z\bm{v} \in \mathcal{A}_n (x) + \ot{\bm{\delta}}x \right\},
    \end{equation*}
    which gives the claimed description of $\Fix \widetilde{\mathcal{T}}$.

    If $\ot Z\bm{v} \in \mathcal{A}_n(x) + \ot{\bm{\delta}}x$ holds, one particularly has
    $$(\ot Z\bm{v})_1\in A_1(x)+\delta_1 x,$$
    which readily gives~\cref{eq: x from v}. To conclude, notice that $\bm{1}_n^*Z=\bm{0}^*_n$, $\bm{1}_n^*\bm{\delta}=0$, and $\ot{(\bm{1}_n^*)}\mathcal{A}_n=\sum_{i=1}^n A_i$. Applying $\ot{(\bm{1}_n^\ast)}$ on both sides of the inclusion $\ot Z\bm{v} \in \mathcal{A}_n(x_{\bm{v}}) + \ot{\bm{\delta}}x_{\bm{v}}$ yields $0\in\sum_{i=1}^n A_i (x_{\bm{v}})$.
\end{proof}

\begin{remark}\label{rem: fixed points}
    We know that the operators $\mathcal{C}^*$ and $(\mathcal{M}+\mathcal{A})^{-1}\mathcal{C}$ bijectively act between $\Fix \mathcal{T}$ and $\Fix \widetilde{\mathcal{T}}$ (recall \cref{fact: Fix T bijection}). In our present graph splitting setting, this can be refined as follows:
    \begin{enumerate}[label=(\roman*)]
        \item If $(\bm{x},\bm{v})\in\Fix \mathcal{T}$, then by \cref{lem: Fix T general}, $\bm{x}=\ot{(\bm{1}_n)}x$ for some $x\in\zer(\sum_{i=1}^n A_i)$ and, consequently, $\mathcal{C}^*(\ot{(\bm{1}_n)}x,\bm{v})=\bm{v}\in\Fix\widetilde{\mathcal{T}}$. From this, we conclude that $\Fix \mathcal{T}\subseteq \Delta_n\times \Fix \widetilde{\mathcal{T}}$. Thus, the operator $\mathcal{C}^*$ can be interpreted as a restriction on the second coordinate of~$\Fix \mathcal{T}$, which is $\Fix \widetilde{\mathcal{T}}$.

        \item\label{rem: MAC} The application of $(\mathcal{M}+\mathcal{A})^{-1}\mathcal{C}$ can be depicted as some kind of parametrization of $\Fix\widetilde{\mathcal{T}}$. Concisely, from \cref{lem: Fix T general}, $\bm{v}\in\Fix \widetilde{\mathcal{T}}$ is equivalent to say that there exists a unique $x_{\bm{v}}\in\zer(\sum_{i=1}^n A_i)$ satisfying the inclusion $\ot Z\bm{v}\in\mathcal{A}_n(x)+\ot{\bm{\delta}}x$. Then, the value of $x$ is parametrized by the choice of $\bm{v}$, and the map $(\mathcal{M}+\mathcal{A})^{-1}\mathcal{C}$ is described as $\bm{v}\mapsto(\bm{1}_nx_{\bm{v}},\bm{v})$.
    \end{enumerate}
\end{remark}

As seen in \cref{lem: Fix T general}, to deduce a closed expression for the sets $\Fix \mathcal{T}$ and $\Fix \widetilde{\mathcal{T}}$, one must solve the inclusion $\ot{Z}\bm{v} \in \mathcal{A}_n(x) + \ot{\bm\delta}x$. In the remainder of this section, we shall show that $\ot{Z^\dagger}$ can be applied to this inclusion to isolate the vector $\bm{v}\in\Hi^{n-1}$, providing a more explicit description of the sets of fixed points. To this end, we introduce the following vector:
\begin{equation}
\bm{\alpha}:={Z^\dagger}\bm{\delta}\in\R^{n-1}.\label{eq: alpha}
\end{equation}

\begin{proposition}\label{prop: v new constraints}
    Let $Z\in\mathbb{R}^{n\times (n-1)}$ be an onto decomposition of $\Lap(G')$ and let $\bm{\delta}\in\R^n$ be the degree balance of $G$. Then:

    \begin{enumerate}[label=(\roman*)]
        \item \label{prop: alpha} The vector $\bm{\alpha}$ is the unique vector satisfying $Z\bm{\alpha}=\bm{\delta}$. Furthermore, $\bm{\alpha}\neq0$.
        \item \label{prop: equiv sol no alpha} For all $x\in\Hi$ and $\bm{e}\in\Hi^{n-1}$, it holds
        \[\ot{Z}\bm{e}\in\mathcal{A}_n(x)\Leftrightarrow \bm{e}\in \ot{Z}^\dagger\left(\mathcal{A}_n(x)\cap \Delta_n^\perp\right).\]
        \item \label{prop: equiv sol} For all $x\in\Hi$ and $\bm{v}\in\Hi^{n-1}$, it holds
    \[\ot{Z}\bm{v}\in\mathcal{A}_n(x)+\ot{\bm{\delta}} x\Leftrightarrow \bm{v}\in \ot{Z}^\dagger\left(\mathcal{A}_n(x)\cap \Delta_n^\perp\right)+\ot{\bm{\alpha}} x.\]
    \end{enumerate}
\end{proposition}
\begin{proof}
    \ref{prop: alpha}: Recall that $ZZ^\dagger=P_{\ran Z}$ (by \cref{def: Moore--Penrose inverse}) and $\ran Z=(\ker Z^\ast)^\perp=\spa\{\bm{1}_n\}^\perp$ (by \cref{fact: Zdecom}). Combining both, we get
\begin{equation}\label{eq: Fix ZZdagger}
    \Fix ZZ^\dagger=\spa\{\bm{1}_n\}^\perp.
\end{equation}
Since $\bm{\delta}\in\spa\{\bm{1}_n\}^\perp$, it holds that $Z\bm{\alpha}=ZZ^\dagger\bm{\delta}=\bm{\delta}$. The uniqueness is a consequence of the injectivity of $Z$, as well as the fact that $\bm{\alpha}\neq 0$, since $\bm{\delta}\neq0$ (e.g., $\delta_1=d_1>0$).

    \ref{prop: equiv sol no alpha}: Since $\ran \ot{Z}=\Delta_n^\perp$ (by \cref{prop: ker Z ran Z}), we deduce that the inclusion $\ot{Z}\bm{e}\in\mathcal{A}_n(x)$ is equivalent to
        \begin{equation}\label{eq: v constraints 2}
           \ot{Z}\bm{e}\in\mathcal{A}_n(x)\cap\Delta_n^\perp.
        \end{equation}

    Note that $\ot{\Z}^\dagger \ot{\Z}=\Id_{\Hi^{n-1}}$, by \cref{rem: Moore--Penrose inverse}. Hence, starting from \cref{eq: v constraints 2} and applying the Moore--Penrose inverse $ \ot{Z}^\dagger$ on both sides, we get $\bm{e}=\ot{Z}^\dagger \ot{Z}\bm{e}\in \ot{Z}^\dagger(\mathcal{A}_n(x)\cap\Delta_n^\perp)$.

    Conversely, starting from the inclusion $\bm{e}\in \ot{Z}^\dagger\left(\mathcal{A}_n(x)\cap\Delta_n^\perp\right)$, applying $\ot{Z}$ on both sides and using the fact that $\ot{Z}\ot{Z}^\dagger=P_{\ran\ot{Z}}=P_{\Delta_n^\perp}$, we obtain $\ot{Z}\bm{e}\in \ot{Z}\ot{Z}^\dagger\left(\mathcal{A}_n(x)\cap\Delta_n^\perp\right)=\mathcal{A}_n(x)\cap\Delta_n^\perp$,
    reaching \cref{eq: v constraints 2}.

    \ref{prop: equiv sol}: By \ref{prop: alpha}, $\bm{\delta}=Z\bm{\alpha}$. Hence $\ot{Z}\bm{v}\in\mathcal{A}_n(x)+\ot{\bm{\delta}} x$ is equivalent to $\ot{Z}(\bm{v}-\ot{\bm{\alpha}} x)\in\mathcal{A}_n(x)$, which, by \ref{prop: equiv sol no alpha}, is equivalent to $\bm{v}-\ot{\bm{\alpha}} x\in\ot{Z^\dagger}\left(\mathcal{A}_n(x)\cap \Delta_n^\perp\right)$, proving the claim.
\end{proof}

\begin{remark}\label{rem: w zeros}
Putting together \cref{lem: Fix T general} and \crefpart{prop: v new constraints}{prop: equiv sol}, we obtain a description of the vectors $\bm{v}$ in $\Fix \widetilde{\mathcal{T}}$ as follows:
$$\bm{v}=\ot{\bs{\alpha}}x+\ot{Z}^\dagger\bm{a}, \quad\text{with } \bm{a}\in \mathcal{A}_n(x)\cap \Delta_n^\perp.$$
The latter means that $\bm{a}=({a}_1,\ldots,{a}_n)$ satisfies
${a}_i\in A_i(x)$, for all $i=1,\ldots, n$, and  $\sum_{i=1}^n {a}_i=0$.
This clearly implies that $x\in\zer(\sum_{i=1}^nA_i)$. Therefore,
$$\Fix \widetilde{\mathcal{T}}= \left\{\ot{\bs{\alpha}}x+\ot{Z}^\dagger\bm{a}\;\Bigl\lvert\; \textstyle x\in\zer(\sum_{i=1}^nA_i), a_i\in A_i(x), \forall i= 1,\ldots, n,  \sum_{i=1}^n a_i=0\right\}.$$
\end{remark}

\section{The case of linear subspaces}\label{sec: 4}

In this section, we particularize the results given in \cref{lem: Fix T general} to the case of normal cone operators of closed linear subspaces. Furthermore, we rewrite the sets of fixed points in such a way that it will be easy to calculate the projections onto them.
Hence, throughout this section we assume that $U_i$ is a closed linear subspace of $\Hi$, for $i=1,\ldots, n$, and
$$A_i(x)=N_{U_i}(x)=\begin{cases} U_i^\perp &\text{if }x\in U_i,\\ \emptyset&\text{otherwise}.\end{cases}$$
We employ the following notation:
\begin{align*}
    U&:=U_1\cap\cdots\cap U_n\subseteq\Hi,\\
    \mathcal{U}&:=U_1\times\cdots\times U_n\subseteq\Hi^n.
\end{align*}

\subsection{Fixed points specializations}

We first particularize \cref{lem: Fix T general}  to the context of linear subspaces. To this end, we introduce the following subspace:
\begin{equation}\label{eq: E}
E:=\ot{Z^\dagger}(\mathcal{U}^\perp\cap \Delta_n^\perp)\subseteq\Hi^{n-1}.
\end{equation}

\begin{lemma}\label{lemma: E closed}
The linear subspace $E$ defined in \cref{eq: E} is closed.
\end{lemma}
\begin{proof}
By~\cref{eq: Fix ZZdagger}, we have that $\Fix \ot{Z}\ot{Z}^\dagger=\Delta_n^\perp$. Thus,
    \begin{equation*}
    \ot{Z}(E)=\ot{Z}\ot{Z^\dagger}(\mathcal{U}^\perp\cap\Delta_n^\perp)=\mathcal{U}^\perp\cap\Delta_n^\perp.
    \end{equation*}
    Now, since $\ot{Z}$ is injective, using the preimage, we get
    $$E=\ot{Z^{-1}}\ot{Z}(E)=\ot{Z^{-1}}(\mathcal{U}^\perp\cap\Delta_n^\perp),$$
    so $E$ is a closed subspace because $\ot{Z}$ is linear and continuous and $\mathcal{U}^\perp\cap\Delta_n^\perp$ is closed.
\end{proof}

\begin{remark}\label{rem: computation E}
    Under our linear subspace setting, one has
    $$\mathcal{A}_n(x)=\begin{cases}\mathcal{U}^\perp & \text{if }x\in U,\\ \emptyset & \text{otherwise}.\end{cases}$$
    Hence, combining \crefpart{prop: v new constraints}{prop: equiv sol no alpha} and \cref{eq: E}, one reaches, for all $\bm e\in\Hi^{n-1}$, the equivalence
    \begin{equation*}
    \bm{e}\in E \Leftrightarrow \ot{Z}\bm{e} \in \mathcal{U}^\perp.
    \end{equation*}
\end{remark}

We are now ready to describe the sets of fixed points of $\mathcal{T}$ and $\widetilde{\mathcal{T}}$ for the linear case. Recall that by \crefpart{fact: proj}{fact: proj_iv} and \cref{fact: proj_v}, the limit points of \cref{alg:T,alg:Ttilde} are the $\mathcal{M}$-projection onto the fixed points of $\mathcal{T}$ and the orthogonal projection onto the fixed points of $\widetilde{\mathcal{T}}$,  respectively.

\begin{proposition}[Specialization for linear subspaces]\label{prop: spec II}
The fixed points of the operators $\mathcal{T}$ and $\widetilde{\mathcal{T}}$ defined in~\cref{eq: T} and~\cref{eq: T tilde} are given by
    \begin{align*}
        \Fix \mathcal{T}&=\ot{(\bm{1}_n,\bm{\alpha})} U\oplus(\{\bm{0}_n\}\times E),\\
        \Fix \widetilde{\mathcal{T}}&=\ot{\bm{\alpha}} U\oplus E,
    \end{align*}
where $\bm{\alpha}$ and $E$ are given by~\cref{eq: alpha,eq: E}, and $\oplus$ denotes the Minkowski sum with orthogonal summands.
\end{proposition}

\begin{proof}
By \cref{lem: Fix T general} and \crefpart{prop: v new constraints}{prop: equiv sol}, the additive decomposition is clear.
It remains to prove the orthogonality of the summands. Notice that it suffices to show the orthogonality for $\Fix \widetilde{\mathcal{T}}$, as the the first $n$ components of the second set in $\Fix \mathcal{T}$ are zero.

    Take $\bm{u}=(u_1,\ldots,u_{n-1})\in \ot{\bm{\alpha}}U$ and $\bm{e}\in E$. Then, by~\cref{eq: E}, there is some ${\bm{b}}=({b}_1,\ldots,{b}_n)\in\mathcal{U}^\perp\cap\Delta_n^\perp$ such that $\bm{e}=\ot{Z^\dagger}{\bm{b}}$. Observe that $u_i\in U\subseteq U_j$ and ${b}_j\in U_j^\perp$, so $u_i\perp{b}_j$, for all $i=1,\ldots,n-1$ and $j=1,\ldots,n$. Thus, we have that
    \begin{align*}
    \langle \bm{u},\bm{e}\rangle=\sum_{i=1}^{n-1}\langle u_i,( \ot{Z^\dagger}{\bm{b}})_i\rangle=\sum_{i=1}^{n-1}\sum_{j=1}^n(\ot{Z^\dagger})_{i,j}\langle u_i,{b}_j\rangle=0,
    \end{align*}
    so the result follows.
\end{proof}

\subsection{Projection onto fixed points}

In this section, we provide explicit expressions for the projection onto $\Fix \widetilde{\mathcal{T}}$ and the $\mathcal{M}$-projection onto $\Fix \mathcal{T}$. Thanks to \cref{prop: spec II}, the fixed point set of $\widetilde{\mathcal{T}}$ is described as an orthogonal sum of linear subspaces. As shown next, this permits to compute the projection onto this set in a simple way, providing the limit point for \cref{alg:Ttilde} as stated in \crefpart{fact: proj}{fact: proj_v}.

\begin{theorem}[Projection onto $\Fix \widetilde{\mathcal{T}}$]\label{th: projection Fix T tilde} We have
    \begin{equation}\label{eq: projection Fix T tilde}
      P_{\Fix \widetilde{\mathcal{T}}}=\ot{\bm{\alpha}}P_U\ot{\bm{\alpha}^\dagger}+P_E.
    \end{equation}
\end{theorem}
\begin{proof}
	By \cref{prop: spec II}, we know that $\Fix \widetilde{\mathcal{T}}=\ot{\bm{\alpha}}U\oplus E$. Note that
    $$\ot{\bm{\alpha}}U=\{(\alpha_1u,\ldots,\alpha_{n-1}u)\in\Hi^{n-1}\mid u\in U\}=U^{n-1}\cap\ran\ot{\bm{\alpha}},$$
    where $U^{n-1}:=\{(u_1,\ldots,u_{n-1})\in\Hi^{n-1} \mid u_i\in U, \text{ for } i=1,\ldots,n-1\}$. This, together with \cref{lemma: E closed}, shows that both summands in $\Fix \widetilde{\mathcal{T}}=\ot{\bm{\alpha}}U\oplus E$ are closed subspaces. Hence, we apply \cite[Proposition~29.6]{bauschke} to compute the projection of $\Fix \widetilde{\mathcal{T}}$ as
    \[P_{\Fix \widetilde{\mathcal{T}}}=P_{\ot{\bm{\alpha}}U\oplus E}=P_{\ot{\bm{\alpha}}U}+P_E.\]

    The rest of the proof of~\cref{eq: projection Fix T tilde} is similar to that of \cite[Lemma~4.3]{bauschke-splitting}, but we provide it for completeness. By \cref{def: Moore--Penrose inverse}, $P_{\ran \ot{\bm{\alpha}}}=\ot{\bm{\alpha}}\ot{\bm{\alpha}^\dagger}$. Since $U$ is a closed linear subspace of $\Hi$, we get that $P_{\ran\ot{\bm{\alpha}}}(U^{n-1})\subseteq U^{n-1}$. Hence, from \cite[Lemma~9.2]{deutsch} we derive that
        \[P_{\ot{\bm{\alpha}}U}=P_{U^{n-1}\cap\ran\bm{\alpha}}=P_{U^{n-1}}P_{\ran \ot{\bm{\alpha}}}=P_{U^{n-1}}\ot{\bm{\alpha}}\ot{\bm{\alpha}^\dagger}=\ot{\bm{\alpha}}P_U\ot{\bm{\alpha}^\dagger},\]
    where the last equality is a consequence of the linearity of $P_U$.
\end{proof}

\begin{remark}
    Note that $\bm{\alpha}\neq 0$, by \crefpart{prop: v new constraints}{prop: alpha}, so recalling \cref{rem: Moore--Penrose inverse}, we can write $\bm{\alpha}^\dagger=\frac{1}{\norm{\bm{\alpha}}^2}\bm{\alpha}^\ast$. Thus, for any $\bm{v}\in\Hi^{n-1}$, we have
    $$\textstyle\ot{\bm{\alpha}}P_U\left(\ot{\bm{\alpha}^\dagger}\bm{v}\right)=\left(\frac{\alpha_i}{\norm{\bm{\alpha}}^2}P_ U\left(\sum_{j=1}^{n-1}\alpha_jv_j\right)\right)_{i=1,\ldots,n-1}$$
    by linearity of the projection.
\end{remark}

Thanks to \cref{fact: projections relations}, we can derive a formula for computing the $\mathcal{M}$-projection onto~$\Fix \mathcal{T}$.

\begin{corollary}[$\mathcal{M}$-projection onto $\Fix \mathcal{T}$]\label{cor: M-projection Fix T} It holds that
    \[P^\mathcal{M}_{\Fix \mathcal{T}}=\ot{\begin{bmatrix}
    \bm{1}_n  \\
    \bm{\alpha}
    \end{bmatrix}}P_U\ot{\bm{\alpha}^\dagger}\mathcal{C}^*+\begin{bmatrix}
    \ot{(\bm{0}_{n\times(2n-1)})}  \\
    P_E\mathcal{C}^*
    \end{bmatrix},\]
    with $\mathcal{M}$ and $\mathcal{C}$ given by~\cref{eq: M}.
\end{corollary}

\begin{proof}
    Let $\bm{w}\in\Hi^n$ and $\bm{v}\in\Hi^{n-1}$ and denote $\bm{y}:=\mathcal{C}^*(\bm{w},\bm{v})=\ot{Z}^*\bm{w}+\bm{v}$. By \cref{fact: projections relations}, $P^\mathcal{M}_{\Fix \mathcal{T}}(\bm{w},\bm{v})=(\mathcal{M}+\mathcal{A})^{-1}\mathcal{C}P_{\Fix \widetilde{\mathcal{T}}}(\bm{y})$. Notice that, by \cref{th: projection Fix T tilde}, one has $P_{\Fix \widetilde{\mathcal{T}}}(\bm{y})= \ot{\bm{\alpha}} u+\bm{e}$, with $u:=P_U(\ot{\bm{\alpha}^\dagger}\bm{y})$ and $\bm{e}:=P_E(\bm{y})$. Then, by \crefpart{rem: fixed points}{rem: MAC}, we deduce that
    $$P^\mathcal{M}_{\Fix \mathcal{T}}    \begin{bmatrix}
        \bm{w} \\ \bm{v}
    \end{bmatrix}=\begin{bmatrix}\ot{(\bm{1}_n)}u\\ \ot{\bm{\alpha}} u+\bm{e}\end{bmatrix},$$
    so the conclusion is achieved.
\end{proof}

\subsection{Algorithmic consequences}\label{sec: 5}

Gathering all the expressions derived in previous sections, we now state our main results about the sequences and limit points of \cref{alg:T,alg:Ttilde} under the current setting of linear subspaces.

\begin{theorem}[Strong convergence of \cref{alg:T}]\label{th: convergence expanded}
Suppose that \cref{as: graphs} holds with $A_i=N_{U_i}$, where $U_i$ is a closed linear subspace of $\Hi$, for $i=1,\ldots,n$. Let $(\bm{x}^k)_{k\in\N}$, $(\bm{w}^k)_{k\in\N}$ and $(\bm{v}^k)_{k\in\N}$ be the sequences generated by \cref{alg:T} with initial points  $\bm{w}^0=(w^0_1,\ldots,w^0_n)\in\Hi^{n}$ and $\bm{v}^0=(v_1^0,\ldots,v_{n-1}^0)\in\Hi^{n-1}$ and a sequence of relaxation parameters $(\theta_k)_{k\in\N}$ identical to some constant $\theta\in{]0,2[}$. Let
\begin{align*}
\overline u  &:= \tfrac{1}{\|\bm{\alpha}\|^2}P_U\left(\ot{\bs{\alpha}}^*\left(\ot{Z^*}\bm{w}^0 + \bm{v}^0\right)\right)=\tfrac{1}{\|\bm{\alpha}\|^2}P_U\left(\ot{\bs{\delta}}^*\bm{w}^0 + \ot{\bs{\alpha}}^*\bm{v}^0\right),\\
\overline{\bm{e}}  &:= P_E\left( \ot{Z}^*\bm{w}^0 +\bm{v}^0 \right).
\end{align*}
Then, the following assertions hold:

\begin{enumerate}[label=(\roman*)]
\item For each $i=1,\ldots,n$, the shadow sequence  $(x_i^k)_{k\in\N}$ and the governing sequence $(w_i^k)_{k\in\N}$ converge strongly to $\overline{u}$.
\item  For each $i=1,\ldots,n-1$, the governing sequence $(v_i^k)_{k\in\N}$ converges strongly to $\alpha_i\overline{u}+\bar{e}_i$.
\end{enumerate}
\end{theorem}
\begin{proof}
First note that, in view of \cref{prop: operator T}, \cref{alg:T} is an instance of iteration \cref{eq: algorithm T}  generated by
\begin{equation*}
    \begin{bmatrix}
        \bm{w}^{k+1} \\ \bm{v}^{k+1}
    \end{bmatrix}=(1-\theta)\begin{bmatrix}
        \bm{w}^{k} \\ \bm{v}^{k}
    \end{bmatrix} + \theta \mathcal{T}\begin{bmatrix}
        \bm{w}^k \\ \bm{v}^{k}
    \end{bmatrix},
\end{equation*}
where $\mathcal{T}$ is the operator defined in \cref{eq: T}. Since $\mathcal{A}$ is a linear relation, we get from \crefpart{fact: proj}{fact: proj_vi} that both $(\bm{w}^{k}, \bm{v}^{k})_{k\in\N}$ and $(\mathcal{T}(\bm{w}^{k}, \bm{v}^{k}))_{k\in\N}$ converge strongly to $(\overline{\bm{w}},\overline{\bm{v}})\in\Fix\mathcal{T}$. Observe that, by \cref{prop: operator T}, it holds
$$\mathcal{T}\begin{bmatrix}
            \bm{w}^k \\ \bm{v}^k
        \end{bmatrix}=\begin{bmatrix}
            \bm{x}^{k+1}\\ \bm{v}^{k}+ \ot{Z^\ast}(\bm{w}^{k}-2\bm{x}^{k+1})
        \end{bmatrix}, \quad\forall k\in\N,$$
and, hence, $\bm{x}^k\to \overline{\bm{w}}$. Finally, using \crefpart{fact: proj}{fact: proj_iv} together with \cref{cor: M-projection Fix T} and noting that  $\ot{\bm{\alpha}}^*\ot{Z}^*=\ot{\bm{\delta}}^*$, by \crefpart{prop: v new constraints}{prop: alpha}, we obtain that
\begin{equation*}
\begin{bmatrix}\overline{\bm{w}}\\ \overline{\bm{v}}\end{bmatrix} = P^\mathcal{M}_{\Fix \mathcal{T}}\begin{bmatrix}\bm{w}^{0}\\ \bm{v}^{0}\end{bmatrix}=\begin{bmatrix}\ot{(\bm{1}_n)}\overline{u}\\ \ot{\bm{\alpha}} \overline{u}+\overline{\bm{e}}\end{bmatrix},
\end{equation*}
and the result is proved.
\end{proof}

\begin{theorem}[Strong convergence of \cref{alg:Ttilde}]\label{th: convergence reduced}
Suppose that \cref{as: graphs} holds with $A_i=N_{U_i}$, where $U_i$ is a closed linear subspace of $\Hi$, for $i=1,\ldots,n$.
Let $(\bm{x}^k)_{k\in\N}$ and  $(\bm{v}^k)_{k\in\N}$ be the sequences generated by \cref{alg:Ttilde} with initial point $\bm{v}^0=(v_1^0,\ldots,v_{n-1}^0)\in\Hi^{n-1}$ and a sequence of relaxation parameters $(\theta_k)_{k\in\N}$ identical to some constant $\theta\in{]0,2[}$. Let
\begin{align*}
\widetilde{u} & := \tfrac{1}{\|\bm{\alpha}\|^2}P_U\left(\ot{\bs{\alpha}}^*\bm{v}^0\right),\\
\widetilde{\bm{e}} & := P_E\left(\bm{v}^0 \right).
\end{align*}
Then, the following assertions hold:

\begin{enumerate}[label=(\roman*)]
\item \label{th: convergence reduced x} For each $i=1,\ldots,n$, the shadow sequence $(x_i^k)_{k\in\N}$ converge strongly to $\widetilde{u}$.
\item \label{th: convergence reduced v}  For each $i=1,\ldots,n-1$, the governing sequence $(v_i^k)_{k\in\N}$ converges strongly to $\alpha_i\widetilde{u}+\widetilde{e}_i$.
\end{enumerate}
\end{theorem}
\begin{proof}
A similar argument to that of \cref{th: convergence expanded} but using \cref{prop: operator T tilde} for iteration \cref{eq: algorithm T tilde} defined by $\widetilde{\mathcal{T}}$ in \cref{eq: T tilde} gives
\begin{equation*}
\bm{v}^k \to  P_{\Fix \widetilde{\mathcal{T}}}(\bm{v}^0)=\ot{\bs{\alpha}}^*\widetilde{u}+\widetilde{\bm{e}},
\end{equation*}
thanks to \cref{th: projection Fix T tilde}, which proves claim  \labelcref{th: convergence reduced v}.

To prove \labelcref{th: convergence reduced x}, set $\bm{u}^0:=(\bm{0}_{\Hi^{n}}, \bm{v}^0)\in\Hi^{2n-1}$. Since $\bm{v}^0=\mathcal{C}^*\bm{u}^0$, \crefpart{fact: proj}{fact: proj_iii} asserts that
\begin{equation*}
(\mathcal{M}+\mathcal{A})^{-1}\mathcal{C}\bm{v}^k=\mathcal{T}(\bm{u^k}), \quad\text{for all } k\in\N,
\end{equation*}
where $(\bm{u}^k)_{k\in\N}$ is the sequence generated by \cref{eq: algorithm T} with $\mathcal{T}$ given by \cref{eq: T}. On the one hand, \cref{fact: proj} and \cref{cor: M-projection Fix T} give us that
\begin{equation*}
\mathcal{T}(\bm{u^k}) \to  P^\mathcal{M}_{\Fix \mathcal{T}}\begin{bmatrix}\bm{0}_{\Hi^n}\\ \bm{v}^{0}\end{bmatrix}=\begin{bmatrix}\ot{(\bm{1}_n)}\widetilde{u}\\ \ot{\bm{\alpha}} \widetilde{u}+\widetilde{\bm{e}}\end{bmatrix}.
\end{equation*}
On the other hand, using \cref{lem:invMA}, we note that the first $n$ components of $(\mathcal{M}+\mathcal{A})^{-1}\mathcal{C}\bm{v}^k$ coincide with $\bm{x}^{k+1}$. Then, $\bm{x}^k\to \ot{(\bm{1}_n)}\widetilde{u}$, as claimed.
\end{proof}

\begin{remark}
Note that if one chooses the initial point $\bm{w}^0\in\Hi^n$ in \cref{alg:T} satisfying $$w_1^0=w_2^0=\cdots=w_n^0\in\Hi,$$
then both \cref{alg:T,alg:Ttilde} converge to the same limit points. Indeed, since $\bm{w}^0\in\Delta_n = \ker{\ot{Z}^*}\subseteq\ker{\ot{\bs{\delta}}^*}$, then  $\ot{\bs{\delta}}^*\bm{w}^0=\ot{Z}^*\bm{w}^0=0$ and the limit points $\overline u$ and $\overline{\bm{e}}$ in \cref{th: convergence expanded} respectively coincide with $\widetilde u$ and $\widetilde{\bm{e}}$ in \cref{th: convergence reduced}.
\end{remark}

In the remainder of this section we analyze some specific algorithms using the results previously derived. We have shown in \cref{th: convergence expanded} that the shadow sequence $(\bm{x}^k)_{k\in\N}$ generated by \cref{alg:T} converges to $\ot{(\bm{1}_{n})}\overline{u}$ and the governing sequence $(\bm{v}^k)_{k\in\N}$ to $\ot{\bm{\alpha}}\overline{u}+P_E\left( \ot{Z}^*\bm{w}^0 +\bm{v}^0 \right)$, with $\overline u= \|\bm{\alpha}\|^{-2}P_U\left(\ot{\bs{\delta}}^*\bm{w}^0 + \ot{\bs{\alpha}}^*\bm{v}^0\right)$, while \cref{th: convergence reduced} proves that the shadow sequence determined by \cref{alg:Ttilde} converges to $\ot{(\bm{1}_{n})}\widetilde{u}$ and the governing sequence to $\ot{\bm{\alpha}}\widetilde{u}+P_E\left(\bm{v}^0 \right)$, with $\widetilde{u} = \|\bm{\alpha}\|^{-2}P_U\left(\ot{\bs{\alpha}}^*\bm{v}^0\right)$. These limit points heavily depend on the vector $\bm{\alpha}$ and the subspace~$E$, given by~\cref{eq: alpha,eq: E}. In what follows we analyze these expressions for some common graphs and deduce the projections $P_{\Fix \widetilde{\mathcal{T}}}$ and $P^\mathcal{M}_{\Fix \mathcal{T}}$ for some specific algorithms, obtaining in this way closed-form expressions for their limit points.

\subsection{Computations of the vector \texorpdfstring{$\bm{\alpha}$}{alpha}}

The value of $\bm{\alpha}$ depends both on the onto decomposition $Z$ of the Laplacian of $G'$ and the degree balance $\bm{\delta}$ of $G$. Hence, there are many possible combinations. To begin our study, we present in \cref{tab: deltas} the degree balance of some specific algorithmic graphs.

\begin{table}[ht!]
\centering
\begin{tabular}{cc}
\toprule
$G$ & $\bm{\delta}$ \\ \midrule
Complete & $\delta_i=n+1-2i$, $1\leq i\leq n$ \\
Sequential & $\delta_1=-\delta_n=1$ and $\delta_i=0$, $1<i<n$ \\
Ring & $\delta_1=-\delta_n=2$ and $\delta_i=0$, $1<i<n$ \\
Parallel up & $\delta_1=n-1$ and $\delta_i=-1$, $1<i\leq n$ \\
Parallel down & $\delta_n=1-n$ and $\delta_i=1$, $1\leq i<n$ \\ \bottomrule
\end{tabular}
\caption{Computation of the degree balance $\bm{\delta}\in\R^n$ for some particular algorithmic graphs $G$ of order $n$.}
\label{tab: deltas}
\end{table}

Now, notice that the onto decomposition $Z$ of $\operatorname{Lap}(G')$ is unique up to orthogonal transformations: if $Z$ and $\widetilde{Z}$ are onto decompositions of $\Lap(G')$, then there exists an orthogonal matrix $O\in\R^{(n-1)\times (n-1)}$ such that $\widetilde{Z}=ZO$ (see \cite[Proposition 2.2]{graph-drs}). We show next that  $\bm{\alpha}$ inherits this property, so the problem of computing $\bm{\alpha}$ can be reduced to the case of a particular onto decomposition.

\begin{proposition}
    Let $Z$ be an onto decomposition of $\Lap(G')$. Let $O\in\R^{(n-1)\times (n-1)}$ be orthogonal and denote $\widetilde{Z}=ZO$. Let $\bm{\alpha}$ and $\widetilde{\bm{\alpha}}$ be computed from \cref{eq: alpha} using $Z$ and $\widetilde{Z}$, respectively. Then $\widetilde{\bm{\alpha}}=O^*\bm{\alpha}$.
\end{proposition}

\begin{proof}
    By \crefpart{prop: v new constraints}{prop: alpha}, we have that $\bm{\delta}=Z\bm{\alpha}=Z(OO^\ast)\bm{\alpha}=\widetilde{Z}(O^*\bm{\alpha})$. Applying again \crefpart{prop: v new constraints}{prop: alpha}, we obtain the desired conclusion.
\end{proof}

The following results show how to compute $\bm{\alpha}$ for two \emph{extreme} graph configurations $(G,G')$, namely, trees and complete graphs. These have the minimum and maximum number of edges allowed under \cref{as: graphs} and also, by \cref{fact: Zdecom}, the onto decompositions of their respective Laplacian has a closed formula, making the computations of $\bm{\alpha}$ simpler.

\begin{proposition}[$\bm{\alpha}$ for trees]\label{prop: alphas tree}
    Suppose that $G=G'$ is a tree of order $n$ and take $Z=\operatorname{Inc}(G)$ (recall \crefpart{fact: Zdecom}{fact: Zdecom_i}). Then $\bm{\alpha}=\bm{1}_{n-1}$.
\end{proposition}

\begin{proof}
    By \crefpart{prop: v new constraints}{prop: alpha},
    $\bm{\alpha}$ is the unique vector such that
    $Z\bm{\alpha}=\bm{\delta}$.
    However, by definition of the incidence matrix, $Z\bm{1}_{n-1}=\Inc(G)\bm{1}_{n-1}=\bm{\delta}$.
\end{proof}

\begin{proposition}[$\bm{\alpha}$ for complete graphs]\label{cor: nZZd = Lap}
    Let $G=G'$ be the complete graph of order $n$ and let $Z$ be an onto decomposition of $\operatorname{Lap}(G)$. Then
    $Z^\dagger = \frac{1}{n}Z^\ast$.
    Moreover, if $Z$ is taken as in \cref{fact: Zdecom}, then
    $$\alpha_j=\frac{1}{t_j}=\sqrt{\frac{(n-j)(n-j+1)}{n}},\quad\text{for all }j=1,\ldots,n-1.$$
\end{proposition}

\begin{proof}
    Note that $\operatorname{Lap}(G)=nI_n-J_n$, where $J_n$ denotes the square matrix of order $n$ whose entries are all one. By \crefpart{fact: Moore--Penrose}{fact: kernel} and~\cref{fact: Zdecom}, one has $\ker Z^\dagger=\ker Z^\ast=\operatorname{span}\{\bm{1}_n\}$, so $Z^\dagger J_n=0$. Then, $$Z^*=Z^\dagger ZZ^*=Z^\dagger\operatorname{Lap}(G)=Z^\dagger(nI_n-J_n)=nZ^\dagger-Z^\dagger J_n=nZ^\dagger.$$

    Now, suppose that $Z$ is given by \cref{eq: Z complete}. Then, for all $j=1,\ldots,n-1$,
            \[\alpha_j=\left(Z^\dagger\bm{\delta}\right)_j=\left(\frac{1}{n}Z^\ast\bm{\delta}\right)_j=\frac{1}{n}\left((n-j)t_j\delta_j - t_j\sum_{i=j+1}^n\delta_i\right)=\frac{t_j}{n}\left((n-j)\delta_j-\sum_{i=j+1}^n\delta_i\right).\]
    Since $\delta_i=n+1-2i$ (see \cref{tab: deltas}), we have
    $$\sum_{i=j+1}^n\delta_i=\sum_{i=j+1}^n(n+1-2i)=(n+1)(n-j)-2\sum_{i=j+1}^n i=j(j-n),$$
    so
    \[\alpha_j=\frac{t_j}{n}\big((n+1-2j)(n-j)-j(j-n)\big)=t_j\frac{(n-j)(n-j+1)}{n}=\frac{t_j}{t_j^2}=\frac{1}{t_j},\]
    which gives $\alpha_j=\frac{1}{t_j}=\sqrt{\frac{(n-j)(n-j+1)}{n}}$.
\end{proof}

\subsection{Computations of the subspace \texorpdfstring{$E$}{E}}
In this subsection, we derive the subspaces $E=\ot{Z^\dagger}(\mathcal{U}^\perp\cap \Delta_n^\perp)$ from \cref{eq: E} for the graphs in \cref{tab: deltas}.

\paragraph{Complete} Recall that \crefpart{fact: Zdecom}{fact: Zdecom_ii} provides an expression for an onto decomposition $Z$ of the Laplacian. Furthermore, by  \cref{cor: nZZd = Lap} we have that $Z^\dagger=\tfrac{1}{n}Z^*$. Thus, we obtain
\begin{equation}\label{eq: E complete}
E=\left\{\mathbf{e}\in\Hi^{n-1}\;\Biggl\lvert\; \begin{array}{l}
e_j = t_j\left((n-j+1)u_j+\sum_{k=1}^{j-1}u_k\right),\text{ for all } j=1,\ldots,n-1;\\
 \text{with } u_j\in U_j^\perp\text{ for }j=1,\ldots,n-1\text{ and }\sum_{j=1}^{n-1}u_j\in U_n^\perp
\end{array}\right\}.
\end{equation}
When $n=2$, this expression simplifies to $E=U_1^\perp\cap U_2^\perp$. This is the case in the Douglas--Rachford algorithm, analyzed in the next subsection. In general, we deduce the inclusion
\begin{equation}\label{eq:Einclusion}
E\subseteq U_1^\perp\times(U_1^\perp+U_2^\perp)\times\cdots\times(U_1^\perp+\cdots +U_{n-2}^\perp)\times((U_1^\perp+\cdots +U_{n-1}^\perp)\cap U_n^\perp).
\end{equation}

\paragraph{Sequential} This graph is a tree, so we can take $Z$ as the incidence matrix (see \crefpart{fact: Zdecom}{fact: Zdecom_i}), which is given by: $Z_{j,j}=1$, $Z_{j+1,j}=-1$, for $j=1,\ldots,n-1$, and zero otherwise. Thus, using the relation $\bm{e}\in E\Leftrightarrow \ot{Z}\bm{e}\in\mathcal{U}^\perp$ provided by \cref{rem: computation E}, we deduce
\begin{equation}\label{eq: E sequential}
\begin{aligned}
E&=\{\bm{e}\in\Hi^{n-1}\mid e_1\in U_1^\perp, e_2-e_1\in U_2^\perp,\ldots, e_{n-1}-e_{n-2}\in U_{n-1}^\perp,-e_{n-1}\in U_n^\perp\},
\end{aligned}
\end{equation}
which also satisfies \cref{eq:Einclusion}.

\paragraph{Ring}
The Laplacian of the ring graph is a circulant matrix, whose first row is given by $(2,-1,0,\ldots,0,-1)\in\R^n$.
Hence, an onto decomposition is provided by \crefpart{fact: Zdecom}{fact: Zdecom_iii}, with eigenvalues
$$\lambda_j=2-\cos\left(\frac{2\pi j}{n}\right)-\cos\left(\frac{2\pi j(n-1)}{n}\right)=2\left(1-\cos\left(\frac{2\pi j}{n}\right)\right)=4\sin^2\left(\frac{\pi j}{n}\right),$$
for $j=1,\ldots,n-1$. Therefore, the subspace $E$ can be computed as (recall \cref{eq: E})
$$E=\ot{Z^\dagger}(\mathcal{U}^\perp\cap \Delta_n^\perp),$$
where $Z^\dagger$ is given by \cref{eq: Zdagger}, that is,
$$Z^\dagger_{j,i}=\frac{\cos\left(\frac{2\pi (i-1)j}{n}-\frac{\pi}{4}\right)}{\sqrt{n\left(1- \cos\left(\frac{2\pi j}{n}\right)\right)}}=\frac{\sin\left(\frac{2\pi (i-1)j}{n}+\frac{\pi}{4}\right)}{\sqrt{2n}\sin\left(\frac{\pi j}{n}\right)},$$
and $Z_{i,j}=\lambda_j Z^\dagger_{j,i}$, for $j=1,\ldots,n-1,\, i=1,\ldots,n$.

\paragraph{Parallel up}
The graph is also a tree, so we can take $Z$ as the incidence matrix, which has the form
\[Z=\begin{bmatrix}
\bm{1}_{n-1}^* \\
-I_{n-1}
\end{bmatrix}.\]
Then, $\bm{e}\in E\Leftrightarrow  \ot{Z}\bm{e}\in \mathcal{U}^\perp\Leftrightarrow e_j\in U_{j+1}^\perp$ for all $j=1,\ldots,n-1$ and $\sum_{j=1}^{n-1} e_j\in U_1^\perp$. In particular, $\sum_{j=1}^{n-1} e_j\in U_1^\perp$ is equivalent to say that $\bm{e}\in\Delta_{n-1}^\perp+(U_1^\perp\times \{\bm{0}_{n-2}\})$. Therefore,
\begin{equation}\label{eq: E parallel up}
E=(U_2^\perp\times\cdots\times U_n^\perp)\cap (\Delta_{n-1}^\perp+(U_1^\perp\times \{\bm{0}_{n-2}\})).
\end{equation}

\paragraph{Parallel down}
Notice that
\begin{equation}\label{eq: Z parallel down}
    Z=\begin{bmatrix}
        I_{n-1} \\
        -\bm{1}_{n-1}^*
        \end{bmatrix}.
\end{equation}
Hence, $\bm{e}\in E\Leftrightarrow e_j\in U_j^\perp$ for all $j=1,\ldots,n-1$ and $\sum_{j=1}^{n-1} e_j\in U_n^\perp$. This gives

\begin{equation}\label{eq: E parallel down}
    E=(U_1^\perp\times\cdots\times U_{n-1}^\perp)\cap (\Delta_{n-1}^\perp+(\{\bm{0}_{n-2}\}\times U_n^\perp)).
\end{equation}

\subsection{Specific algorithms}

Finally, we can apply \cref{th: projection Fix T tilde} to compute the projection onto $\Fix \widetilde{\mathcal{T}}$ and \cref{cor: M-projection Fix T} to obtain the $\mathcal{M}$-projection onto $\Fix \mathcal{T}$. Let us apply these results to some specific examples of splitting algorithms that appear in the literature, most of which were developed prior to the graph framework.

\paragraph{Douglas--Rachford}
This splitting algorithm, developed in \cite{LM79}, can be described as a graph-scheme by taking $G=G'$ as the unique algorithmic graph of order two, consisting of a unique edge connecting nodes $1$ and $2$. Since it is a tree, we can apply \cref{prop: alphas tree} and deduce that $\alpha=1$. Also, as $G'$ is complete, we have seen in \cref{eq: E complete} that $E=U_1^\perp\cap U_2^\perp$. Hence, by \cref{prop: spec II}, we have
\begin{align*}
        \Fix \mathcal{T}&=\left( U^3\cap\Delta_3\right)\oplus\big(\{0\}^2\times (U_1^\perp\cap U_2^\perp)\big),\\
        \Fix \widetilde{\mathcal{T}}&= (U_1\cap U_2)\oplus (U_1^\perp\cap U_2^\perp).
    \end{align*}
Therefore, for any $v\in\Hi$, we have
\begin{equation}\label{eq: DR}
	P_{\Fix \widetilde{\mathcal{T}}}(v)=P_{U_1\cap U_2}(v)+P_{U_1^\perp\cap U_2^\perp}(v),
\end{equation}
while for any $(w_1,w_2,v)\in\Hi^3$, we have $y:=\mathcal{C}^*(w_1,w_2,v)=w_1-w_2+v$, so \cref{cor: M-projection Fix T} yields
\begin{equation}\label{eq: M-projection DRS}
P^\mathcal{M}_{\Fix \mathcal{T}}\begin{bmatrix}
w_1 \\ w_2 \\ v
\end{bmatrix}=\begin{bmatrix}
P_{U_1\cap U_2}(y) \\
P_{U_1\cap U_2}(y) \\
P_{U_1\cap U_2}(y)+P_{U_1^\perp\cap U_2^\perp}(y)
\end{bmatrix}.
\end{equation}

\begin{remark}
    The $\mathcal{M}$-projection for the Douglas--Rachford was also obtained in \cite[Equations (39a) and (39b)]{bauschke-fixedpoints}, although slightly distinct from \cref{eq: M-projection DRS}. This is due to the choice of the operators $\mathcal{A}$ and $\mathcal{M}$. To see this, let us follow the construction from \cite[Section 5.1]{bauschke-fixedpoints} (originally from \cite[Section 3]{degenerate-ppp}). One defines $\mathcal{C}:=\ot Z$, hence $\mathcal{M}=\ot{\Lap(G')}$, and
    \[\mathcal{A}:=\Diag(A_1,A_2^{-1})+\ot{P(G)}-\ot{\Lap(G')}=\begin{bmatrix}
	A_1 & \Id \\ -\Id & A_2^{-1}
    \end{bmatrix}.\]
    With this, for any $(w_1,w_2)\in\Hi^2$ and denoting by $y:=\mathcal{C}^*(w_1,w_2)=w_1-w_2$, the $\mathcal{M}$-projection given in \cite[Equation (39b)]{bauschke-fixedpoints} is computed as
    \[P_{\Fix\mathcal{T}}^\mathcal{M}(w_1,w_2)=\begin{bmatrix}P_{U_1\cap U_2}(y) \\-P_{U_1^\perp\cap U_2^\perp}(y)\end{bmatrix}.\]
    This formula compared with \cref{eq: M-projection DRS} is reduced by one component and differs in the second component. On the other hand, notice that $\mathcal{C}^*P_{\Fix\mathcal{T}}^\mathcal{M}(w_1,w_2)$ coincides with \cref{eq: DR}. Therefore, while in our work the graph configuration is used to develop the algorithm (so $\mathcal{A}$ and $\mathcal{M}$ are operators in$~\Hi^3$), in \cite[Section 5.1]{bauschke-fixedpoints} a simplified and equivalent version of these operators were chosen, with $\mathcal{A}$ and $\mathcal{M}$ being defined in $\Hi^2$.
\end{remark}

\paragraph{Generalized Ryu}

Although Ryu's scheme was originally proposed in~\cite{ryu20} for 3 maximally monotone operators, it was shown in~\cite{tam2023frugal,graph-drs} that it can be naturally extended for any number of operators. In particular, it was noted in \cite[pp.~1579--1580]{graph-drs} that it can be described as a graph-devised algorithm by letting $G$ to be the complete graph of order $n$ and $G'$ the parallel down. In this case, it is direct to compute $\bm{\alpha}$ using the equation $Z\bm{\alpha}=\bm{\delta}$, by \crefpart{prop: v new constraints}{prop: alpha}. This gives, by~\cref{eq: Z parallel down}, $\alpha_j=\delta_j=n+1-2j$ for all $j=1,\ldots,n-1$ (recall \cref{tab: deltas}).

Hence, for any $\bm{v}\in\Hi^{n-1}$, one has from \cref{th: projection Fix T tilde} that
\begin{align*}
P_{\Fix \widetilde{\mathcal{T}}}(\bm{v})&=\ot{\bm{\alpha}}P_U(\ot{\bm{\alpha}^\dagger}\bm{v})+P_E(\bm{v})\\
&=\left(\left(\frac{n+1-2i}{\norm{\bm{\alpha}}^2}\right)P_ U\left(\textstyle\sum_{j=1}^{n-1}(n+1-2j)v_j\right)\right)_{i=1,\ldots,n-1}+P_E(\bm{v}),
\end{align*}
where
$$\norm{\bm{\alpha}}^2=\sum_{j=1}^{n-1}(n+1-2j)^2=\frac{n-1}{3}\left((n-1)^2+2\right).$$
Further, for any $(\bm{w},\bm{v})\in\Hi^n\times\Hi^{n-1}$, it holds (see \cref{cor: M-projection Fix T}) that
$$P^\mathcal{M}_{\Fix \mathcal{T}}\begin{bmatrix}
\bm{w} \\ \bm{v}
\end{bmatrix}=\begin{bmatrix}
\ot{(\bm{1}_n)}P_U(\ot{\bm{\alpha}^\dagger} \bm{y}) \\
P_{\Fix \widetilde{\mathcal{T}}}(\bm{y})
\end{bmatrix}=\begin{bmatrix}
\ot{(\bm{1}_n)}P_U\left(\frac{1}{\norm{\bm{\alpha}}^2}\sum_{j=1}^{n-1}(n+1-2j)y_j\right) \\
P_{\Fix \widetilde{\mathcal{T}}}(\bm{y})
\end{bmatrix},$$
with $\bm{y}:=\mathcal{C}^*(\bm{w},\bm{v})=(w_j-w_n)_{j=1,\ldots,n-1}+\bm{v}$.

For the case of the original splitting algorithm introduced by Ryu, where $n=3$, we obtain, using also \cref{eq: E parallel down},
$$
    \bm{\alpha}=(2,0)\in\R^2\quad\text{and}\quad
    E=(U_1^\perp\times U_2^\perp)\cap(\Delta_2^\perp+(\{0\}\times U_3^\perp))\subseteq\Hi^2.
$$
Thus, the projections onto the fixed points given above simplify to
\begin{align*}
    P_{\Fix \widetilde{\mathcal{T}}}(v_1,v_2)&=(P_U(v_1),0)+P_E(v_1,v_2),\\
    P^\mathcal{M}_{\Fix \mathcal{T}}(w_1,w_2,w_3,v_1,v_2)&=\left(\frac{1}{2}P_U(y_1),\frac{1}{2}P_U(y_1),\frac{1}{2}P_U(y_1),P_{\Fix \widetilde{\mathcal{T}}}(y_1,y_2)\right),
\end{align*}
with $y_1=w_1-w_3+v_1$ and $y_2=w_2-w_3+v_2$. These expressions coincide with those in~\cite[Corollary~5.6]{bauschke-fixedpoints}.

\paragraph{Malitsky--Tam}

As shown in~\cite[p.~1580]{graph-drs}, this algorithm can be described by a graph configuration $(G,G')$, where $G$ is the ring graph and $G'$ is the sequential graph. Taking into account \cref{tab: deltas} and \crefpart{prop: v new constraints}{prop: alpha}, one can compute $\bm{\alpha}$ as follows:
\[Z\bm{\alpha}=\bm{\delta}\Leftrightarrow\left\{\begin{array}{l}
\alpha_1=\delta_1=2\\
\alpha_{j+1}-\alpha_j=\delta_j=0,\; j=1,\ldots,n-2\\
-\alpha_{n-1}=\delta_{n-1}=-2
\end{array}\right\}\Leftrightarrow\bm{\alpha}=2\cdot\bm{1}_{n-1}.\]
Notice then that $\norm{\bm{\alpha}}^2=4(n-1)$, so $\bm{\alpha}^\dagger=\frac{1}{2(n-1)}\bm{1}_{n-1}^*$. Hence, for any $\bm{v}\in\Hi^{n-1}$, it holds (see \cref{th: projection Fix T tilde}) that
\[P_{\Fix \widetilde{\mathcal{T}}}(\bm{v})=\ot{(\bm{1}_{n-1})}P_ U\left(\textstyle\sum_{j=1}^{n-1}\frac{v_j}{n-1}\right)+P_E(\bm{v}),\]
where $E$ is given by~\cref{eq: E sequential}.

Further, for any $\bm{w}\in\Hi^n$ and $\bm{v}\in\Hi^{n-1}$, letting $\bm{y}:=\mathcal{C}^\ast(\bm{w},\bm{v})=(w_{j}-w_{j+1})_{j=1,\ldots,n-1}+\bm{v}$, we obtain from \cref{cor: M-projection Fix T} that
\[P^\mathcal{M}_{\Fix \mathcal{T}}\begin{bmatrix}
    \bm{w} \\
    \bm{v}
    \end{bmatrix}=\begin{bmatrix}
    \ot{(\bm{1}_n)} \frac{1}{2}P_U\left(\sum_{j=1}^{n-1} \frac{y_j}{n-1}\right) \\
    \ot{(\bm{1}_{n-1})}P_U\left(\sum_{j=1}^{n-1} \frac{y_j}{n-1}\right)+P_E(\bm{y})
    \end{bmatrix}.\]
Both projections coincide with those derived in~\cite[Corollary~5.8]{bauschke-fixedpoints}.

\paragraph{Parallel up}
The graph configuration $(G,G')$ where $G=G'$ are parallel up leads to the splitting method proposed in \cite[Theorem 5.1]{campoy} and~\cite[Equation~(9.2)]{condat2023proximal}. Since these graphs are trees, by \cref{prop: alphas tree}, ${\bm{\alpha}=\bm{1}_{n-1}}$. Thus, for any $\bm{w}\in\Hi^{n}$ and $\bm{v}\in\Hi^{n-1}$, we obtain by \cref{th: projection Fix T tilde,cor: M-projection Fix T} that
\begin{equation}\label{eq: projections parallel up}
    \begin{aligned}
        P_{\Fix \widetilde{\mathcal{T}}}(\bm{v})&=\ot{(\bm{1}_{n-1})}P_ U\left(\textstyle\sum_{j=1}^{n-1} \frac{v_j}{n-1}\right)+P_E(\bm{v}),\\
        P^\mathcal{M}_{\Fix \mathcal{T}}\begin{bmatrix}
            \bm{w} \\
            \bm{v}
            \end{bmatrix}&=\ot{(\bm{1}_{2n-1})}P_U\left(\textstyle\sum_{j=1}^{n-1} \frac{y_j}{n-1}\right)+ \begin{bmatrix}
            \bm{0}_n \\
            P_E(\bm{y})
            \end{bmatrix},\ \text{with }\bm{y}:=\mathcal{C}^*(\bm{w},\bm{v})=\ot{Z^*}\bm{w}+\bm{v},
    \end{aligned}
\end{equation}
where $E$ is given by~\cref{eq: E parallel up}. These expressions improve upon \cite[Theorem~4.6]{bauschke-splitting}, where no explicit formula for the projection was derived.

\paragraph{Parallel down}

Analogously to the previous scheme (see \cite[Equation~(9.3)]{condat2023proximal}), the method is devised by the graph configuration $G=G'$ being parallel down graphs. Again, we have $\bm{\alpha}=\bm{1}_{n-1}$ by \cref{prop: alphas tree}.
Finally, the projections onto the fixed points can be computed as in \cref{eq: projections parallel up}, using in this case the set $E$ given by~\cref{eq: E parallel down}.

\paragraph{Sequential}

This method, developed in \cite[\S~3.1.2]{degenerate-ppp}, is defined by the graph configuration $G=G'$ being sequential graphs, so $\bm{\alpha}=\bm{1}_{n-1}$ by \cref{prop: alphas tree}. Hence, the projections can be computed as in~\cref{eq: projections parallel up} but with $E$ given by~\cref{eq: E sequential}.

\paragraph{Complete}

A new splitting method was introduced in \cite{graph-drs} for $n=3$ through the graph configuration $G=G'$ being complete graphs. A generalized version of this method for any $n$ (including also cocoercive operators) was proposed in \cite{graph-fb}, by considering the onto decomposition $Z$ given in \cref{fact: Zdecom}. Thanks to \cref{cor: nZZd = Lap}, both $Z^\dagger$ and $\bm{\alpha}$ can be explicitly determined, having
$$\norm{\bm{\alpha}}^2=\frac{1}{n}\sum_{j=1}^{n-1}(n-j)(n-j+1)=\frac{n^2-1}{3}.$$
Hence, for any $(\bm{w},\bm{v})\in\Hi^{n}\times\Hi^{n-1}$, we deduce that the projections are given by
$$
    \begin{aligned}
        P_{\Fix \widetilde{\mathcal{T}}}(\bm{v})&=\left(\textstyle\frac{3}{(n^2-1)t_i}P_ U(\textstyle\sum_{j=1}^{n-1} \frac{v_j}{t_j})\right)_{i=1,\ldots,n-1}+P_E(\bm{v}),\\
        P^\mathcal{M}_{\Fix \mathcal{T}}\begin{bmatrix}
            \bm{w} \\
            \bm{v}
            \end{bmatrix}&=\begin{bmatrix}
                \frac{3}{n^2-1}\ot{(\bm{1}_{n})}P_U(\textstyle\sum_{j=1}^{n-1} \frac{y_j}{t_j}) \\
            P_{\Fix \widetilde{\mathcal{T}}}(\bm{y})
            \end{bmatrix},\ \text{with }\bm{y}:=\mathcal{C}^*(\bm{w},\bm{v})=\ot{Z^*}\bm{w}+\bm{v},
    \end{aligned}
$$
where $E$ is given by~\cref{eq: E complete}.\smallskip

For the reader's convenience, we finish this section by summarizing in \cref{tab:alg_alpha_delta} the choice of graphs $(G,G')$ for all the algorithms analyzed, together with the values of $\bm{\alpha}$ and the specific limit point $\widetilde{u}$ of the shadow sequence $(\bm{x}^k)_{k\in\mathbb{N}}$ in \cref{alg:Ttilde}, established in \crefpart{th: convergence reduced}{th: convergence reduced x}. We also illustrate these results in \cref{fig:planes}, where we represent the sequences generated by the six algorithms displayed in \cref{tab:alg_alpha_delta} for the problem of finding the point in the intersection of three planes in $\R^2$, showing the convergence to the limit points obtained by our analysis. For a numerical study comparing their performance in this setting, see~\cite{graphnumerical}.

\begin{table}[ht!]
\centering\small
\begin{tabular}{ccccc}
\toprule
Algorithm & $G$ & $G'$ &${\alpha}_j$ & Limit point of \cref{alg:Ttilde} \\ \midrule
Douglas--Rachford & Sequential & Sequential & $1$ & $P_{U_1\cap U_2}(v)$\\
Generalized Ryu & Complete & Parallel down & $n+1-2j$ & $\frac{3\left((n-1)^2+2\right)}{n-1}P_ U\left(\textstyle\sum_{j=1}^{n-1}\alpha_jv_j\right)$ \\
Malitsky--Tam & Ring & Sequential & $2$ & $\frac12P_ U\left(\textstyle\sum_{j=1}^{n-1}\frac{v_j}{n-1}\right)$ \\
Parallel up & Parallel up & Parallel up & $1$ & $P_ U\left(\textstyle\sum_{j=1}^{n-1} \frac{v_j}{n-1}\right)$ \\
Parallel down & Parallel down & Parallel down & $1$ & $P_ U\left(\textstyle\sum_{j=1}^{n-1} \frac{v_j}{n-1}\right)$ \\
Sequential & Sequential & Sequential & $1$ & $P_ U\left(\textstyle\sum_{j=1}^{n-1} \frac{v_j}{n-1}\right)$ \\
Complete & Complete &  Complete & $\textstyle\sqrt{\frac{(n-j)(n-j+1)}{n}}$ & $\textstyle\frac{3}{n^2-1}P_ U\left(\textstyle\sum_{j=1}^{n-1}\alpha_jv_j\right)$ \\
\bottomrule
\end{tabular}
\caption{Summary of the graph splitting algorithms considered of order $n$}
\label{tab:alg_alpha_delta}
\end{table}

\begin{figure}[htp!]

\centering
\def\wfg{0.33\textwidth}
\def\sfg{0.08\textwidth}
\setlength{\fboxsep}{0pt}
\includegraphics[width=.75\textwidth]{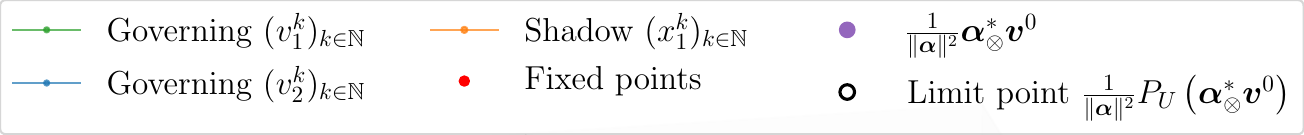}
\vspace{4mm}

  \begin{subfigure}{\wfg}
    \centering
    {\includegraphics[width=\linewidth]{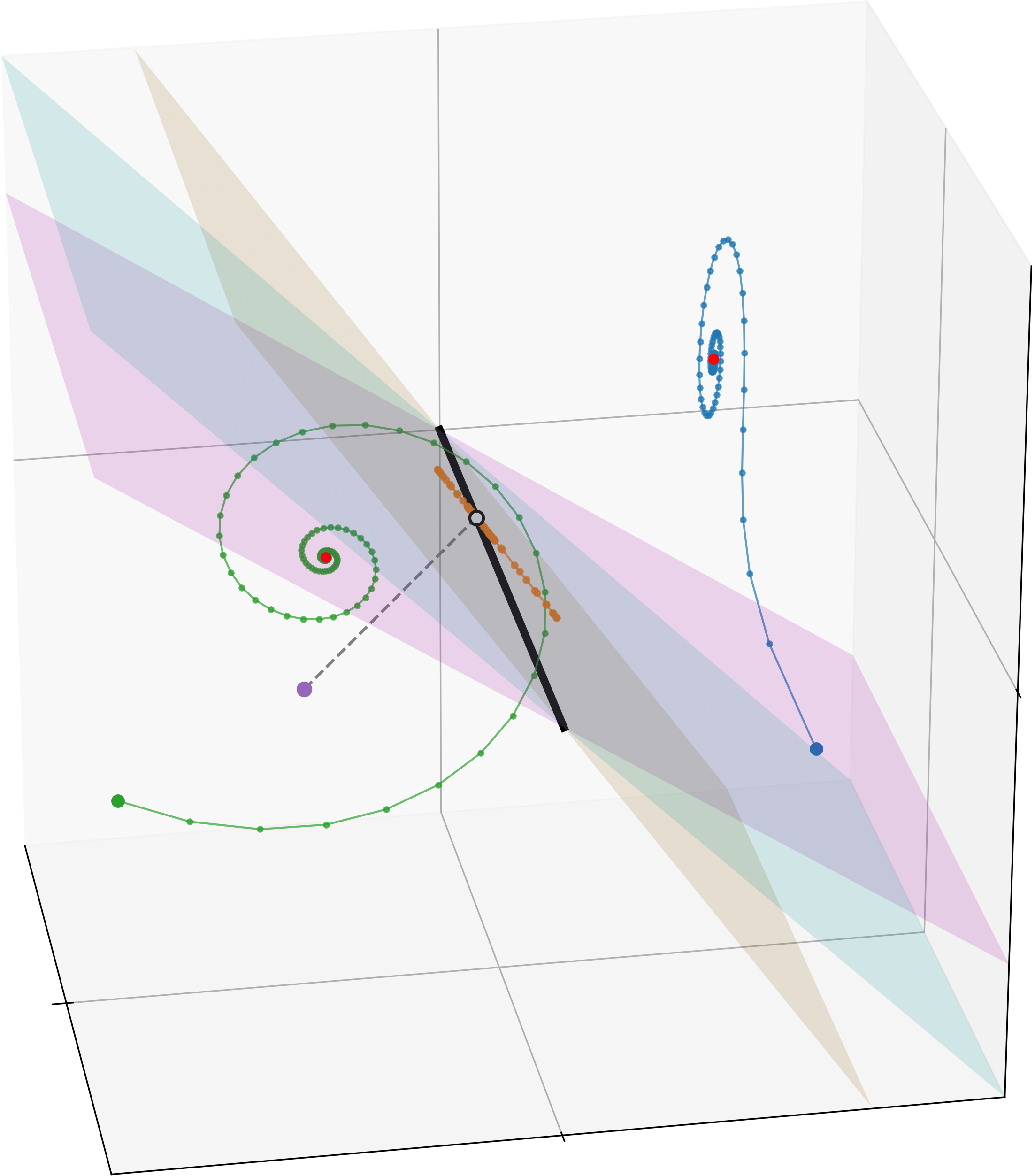}}
    \caption{Generalized Ryu}
  \end{subfigure}\hspace{\sfg}%
  \begin{subfigure}{\wfg}
    \centering
    {\includegraphics[width=\linewidth]{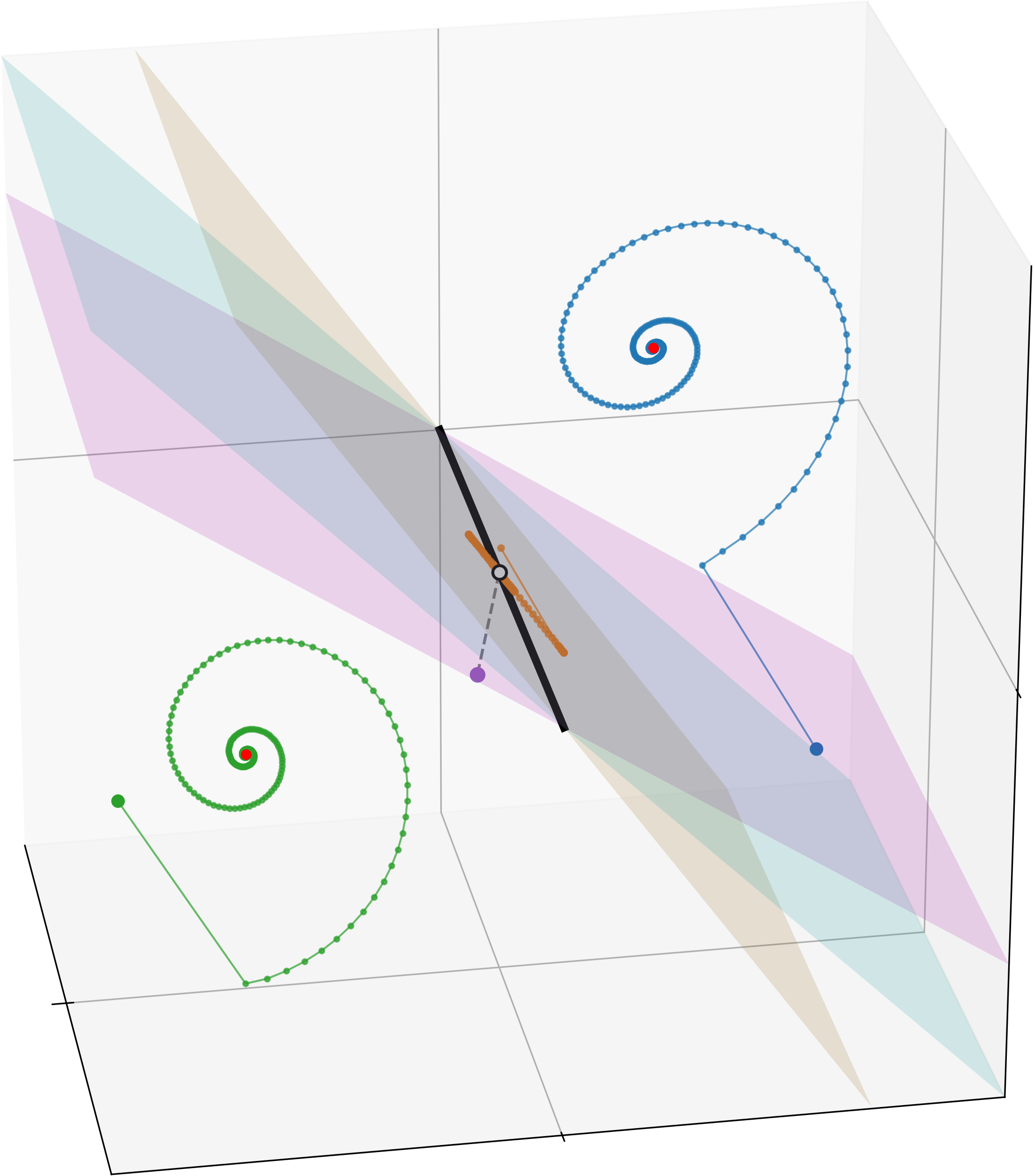}}
    \caption{Makitsky--Tam}
  \end{subfigure}
\vspace{5mm}

  \begin{subfigure}{\wfg}
    \centering
    {\includegraphics[width=\linewidth]{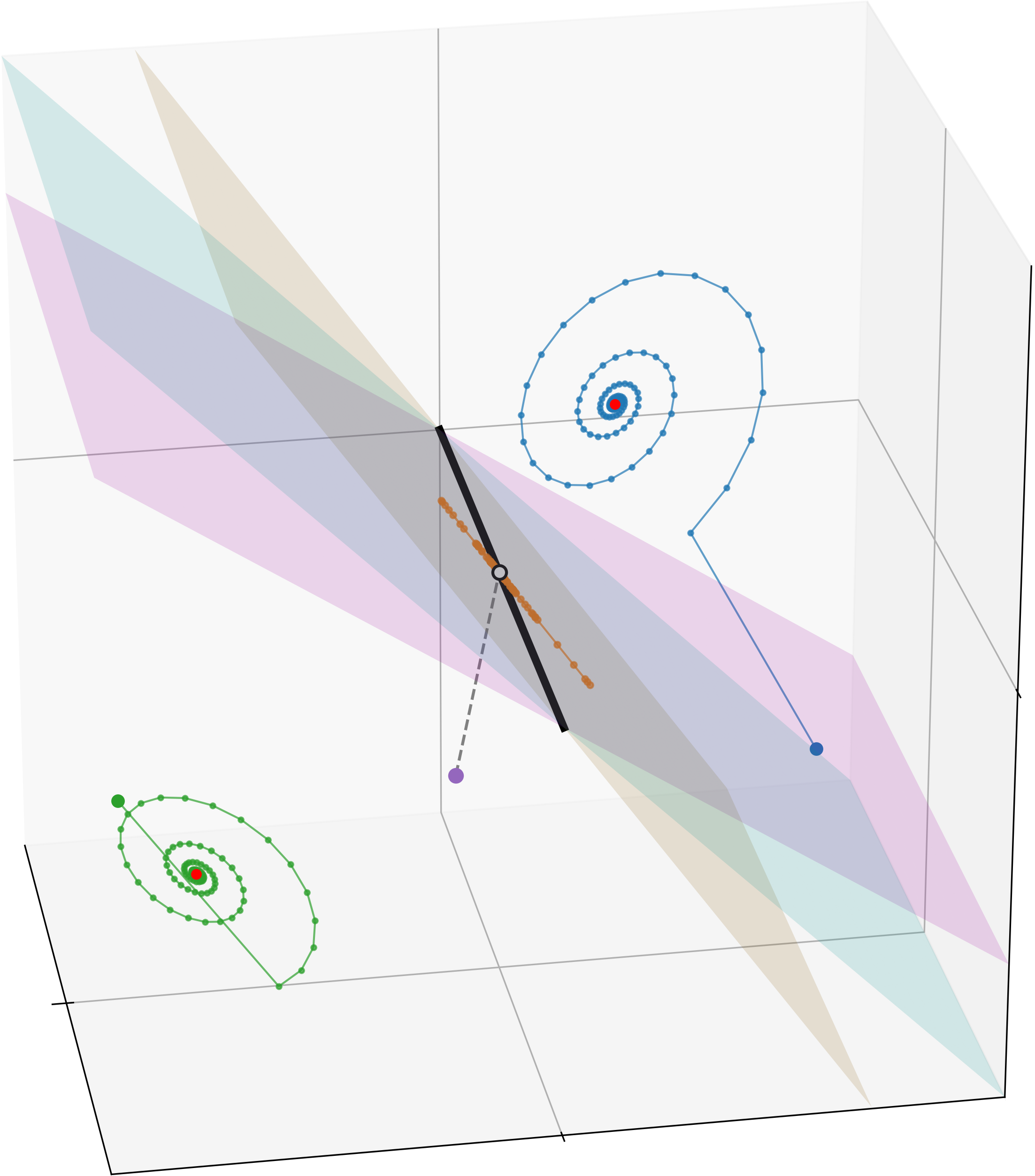}}
    \caption{Parallel up}
  \end{subfigure}\hspace{\sfg}%
  \begin{subfigure}{\wfg}
    \centering
    {\includegraphics[width=\linewidth]{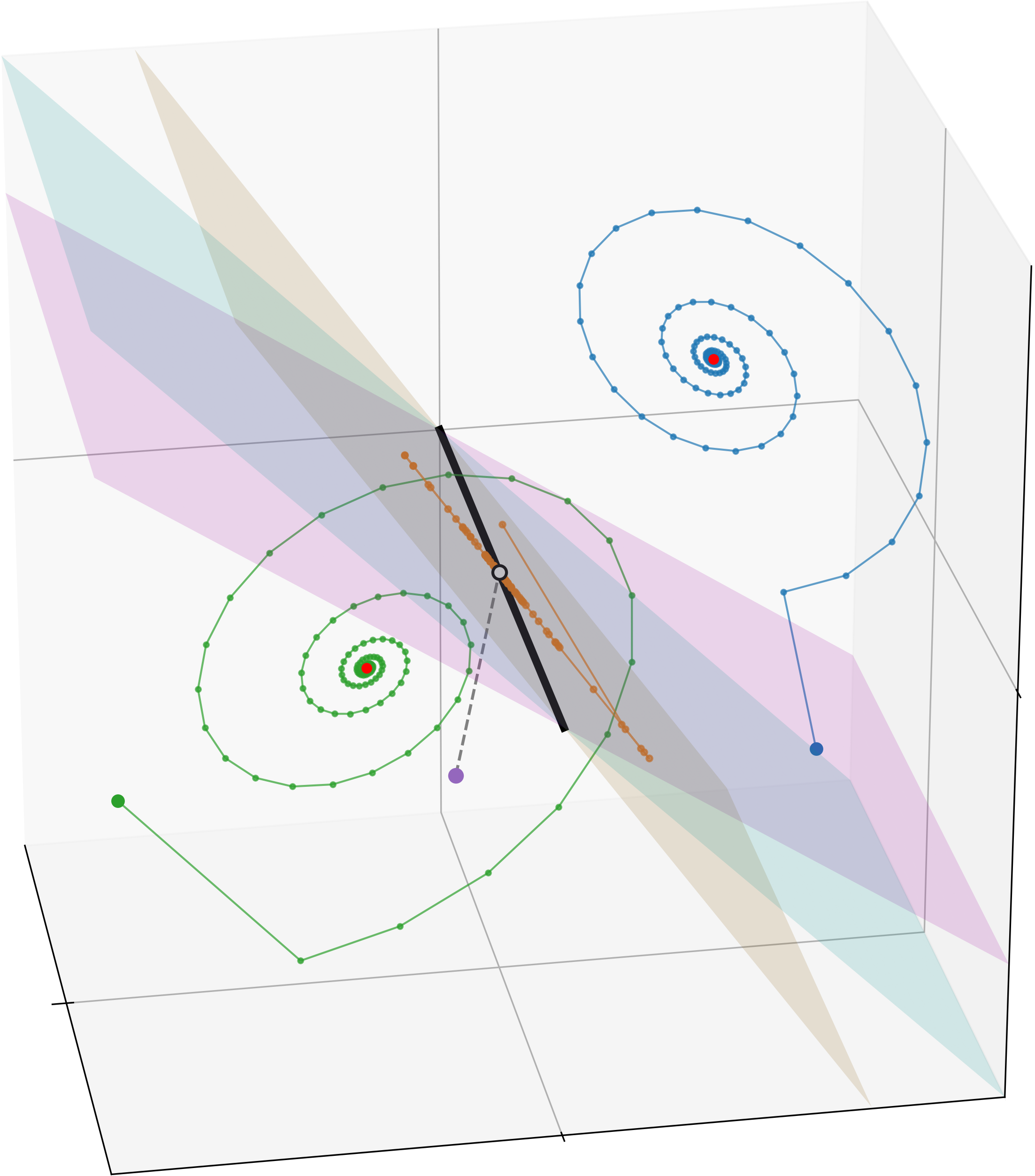}}
    \caption{Parallel down}
  \end{subfigure}
\vspace{5mm}

  \begin{subfigure}{\wfg}
    \centering
    {\includegraphics[width=\linewidth]{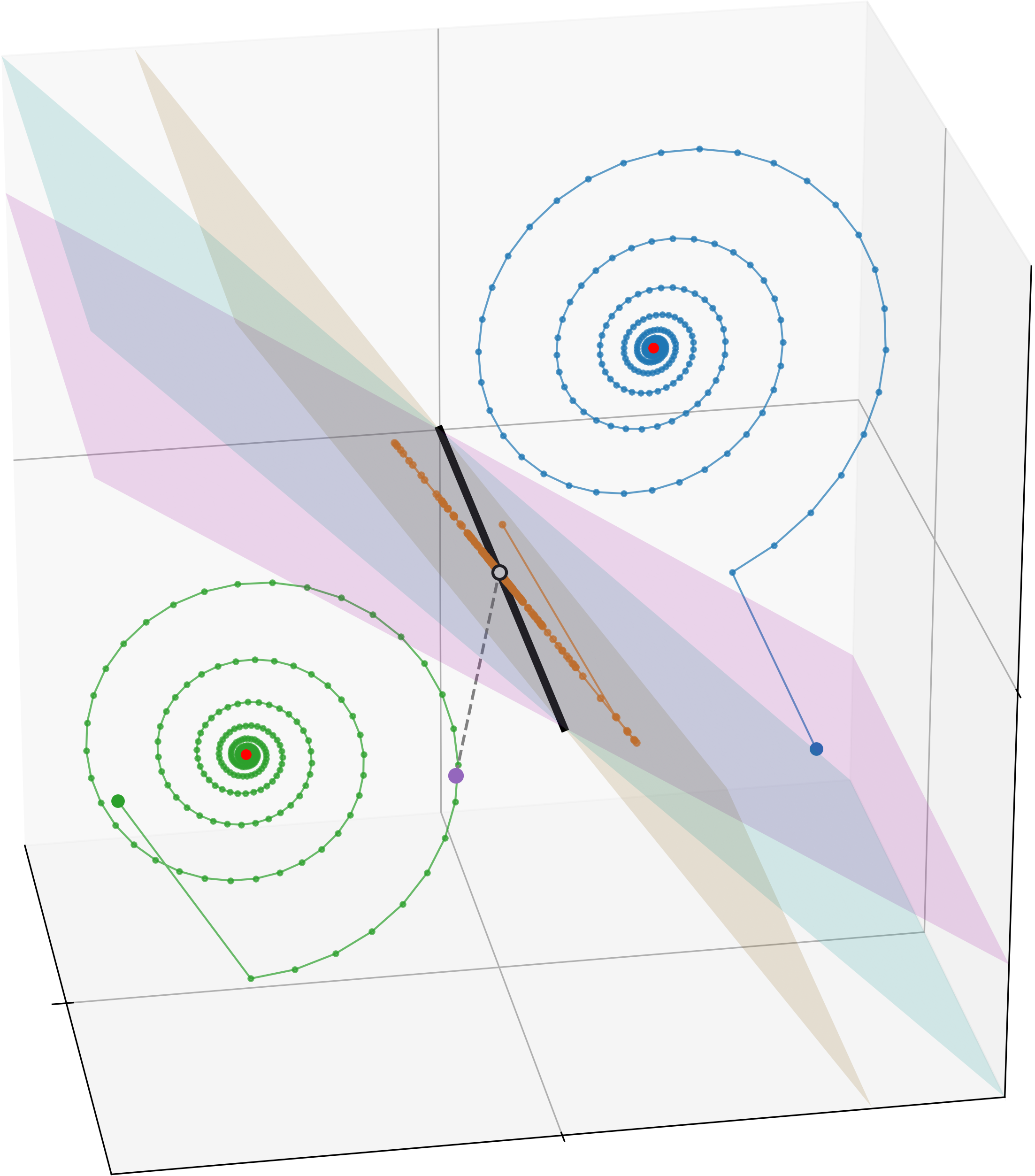}}
    \caption{Sequential}
  \end{subfigure}\hspace{\sfg}%
  \begin{subfigure}{\wfg}
    \centering
    {\includegraphics[width=\linewidth]{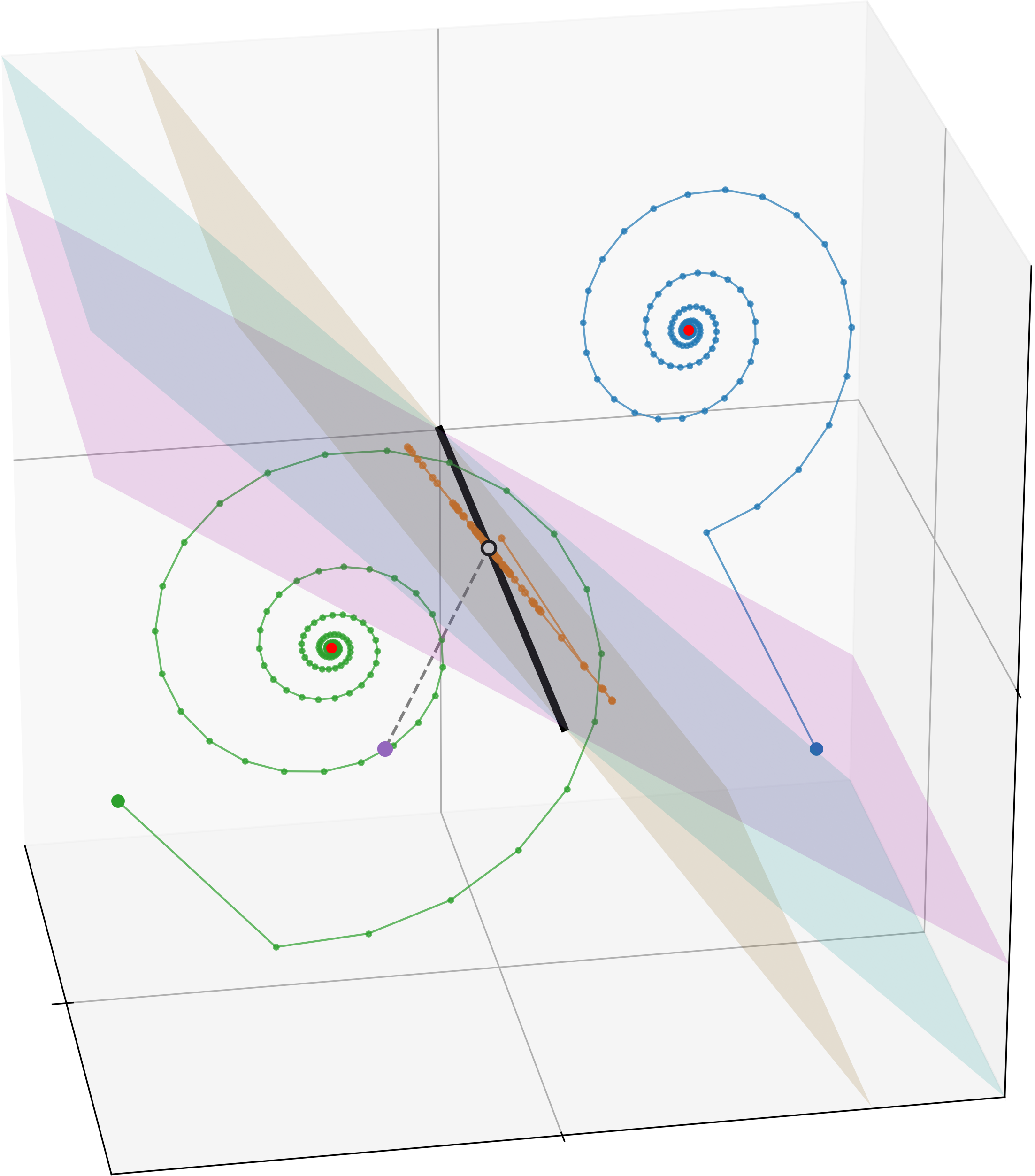}}
    \caption{Complete}
  \end{subfigure}
\vspace{2mm}

\caption{Behavior of six graph splitting algorithms applied to finding a point in the intersection of three planes in $\R^3$. Sequences and limit points are shown for the same pair of starting points}\label{fig:planes}
\end{figure}

\section{Conclusions}\label{sec: concl}

In this work, we have employed the graph-based framework of \cite{graph-drs} to unify and extend the analysis of preconditioned proximal point algorithms carried out in \cite{bauschke-fixedpoints}. By using this unified approach, we have derived general results that provide characterization of the limit points of different splitting algorithms. Overall, our main contributions are:

\begin{itemize}
\item Providing an explicit expression for the set of fixed points of the underlying operators defining the algorithms for general graph configurations.
\item Computing the limit points of an algorithm for a general graph configuration when the involved operators are normal cones of a closed linear subspaces.
\item Generalizing the results of \cite[Section 5]{bauschke-fixedpoints}, as well as conducting analogous results on additional splitting schemes.
\end{itemize}

In conclusion, we have shown how the use of the graph splitting framework streamlines the analysis of existing splitting algorithms, especially in the context of feasibility problems with linear subspaces, allowing to obtain results that apply broadly rather than requiring a separate inspection for each algorithm.

\paragraph{Acknowledgments}
The authors are very thankful to the anonymous referees for their helpful comments, which allowed us to improve the original manuscript.

\bibliographystyle{acm}
\bibliography{bibliography.bib}

\end{document}